\documentclass[11pt, reqno]{amsart}
\usepackage[english]{babel}
\usepackage{amsthm} 
\usepackage{amssymb}
\usepackage{amsmath}
\usepackage{hyperref}
\usepackage{amsfonts}
\usepackage{stmaryrd}
\usepackage{fancyhdr}
\usepackage{graphicx,color}
\usepackage[left=1in, right=1in, top=1in, bottom=1in]{geometry}

\newtheorem{theorem}{Theorem}[section]
\newtheorem{prop}[theorem]{Proposition}

\newtheorem{lemma}[theorem]{Lemma}
\newtheorem{definition}[theorem]{Definition}

\newtheorem{note}[theorem]{Note}
\newtheorem{corollary}[theorem]{Corollary}

\newtheorem{example}[theorem]{Example}

\newcommand{\A}{\mathcal{A}}

\newcommand{\N}{\mathbb{N}}
\newcommand{\cL}{\mathcal{L}}
\newcommand{\vp}{\varphi}
\newcommand{\cP}{\mathcal{P}}

\newcommand{\hf}{\hfill$\Box$} 
\newcommand{\pa}{\partial}
\newcommand{\lv}{\left\vert}
\newcommand{\rv}{\right\vert}
\newcommand{\be}{\begin{equation}}
\newcommand{\ee}{\end{equation}}
\newcommand{\dd}{,{\dots},}
\newcommand{\al}[1]{a_{[\alpha_{#1}]}}
\newcommand{\vl}[1]{V_{[\alpha_{#1}]} f_t}
\newcommand{\la}{\langle}
\newcommand{\ra}{\rangle}
\newcommand{\V}{\mathbf{V}}
\newcommand{\X}{\mathbf{X}}
\newcommand{\Y}{\mathbf{Y}}
\newcommand{\am}{_{\A_m}}

\newcommand{\R}{\mathbb{R}}
\newcommand{\RR}{\mathfrak{R}^{\lv \A_m\rv}}

\begin{document}

\title{Pointwise gradient bounds for degenerate semigroups (of UFG type)}
\author{ D. Crisan and M. Ottobre}

\address{\noindent \textsc{Dan Crisan, Department of Mathematics, Imperial College London, Huxley Building, 180 Queen's Gate,  London SW7 2AZ, UK}} 
\email{d.crisan@imperial.ac.uk}
\address{\noindent \textsc{Michela Ottobre, Department of Mathematics, Heriot-Watt University, Edinburgh EH14 4AS, UK}} 
\email{m.ottobre@hw.ac.uk}


\begin{abstract}
In this paper we consider diffusion semigroups generated by second order differential operators of degenerate type. The operators that we consider {\em do not}, in general, satisfy the H\"ormander condition and are {\em not} hypoelliptic. In particular, instead of working under the H\"ormander paradigm,   we consider  the so-called  UFG condition, introduced by Kusuoka and Strook in the eighties. 
 The UFG condition  is weaker than the uniform H\"ormander condition, the smoothing effect taking place only in certain directions (rather than in every direction, as it is the case when the H\"ormander condition is assumed). Under the UFG condition, Kusuoka and Strook  deduced  sharp \emph{small} time asymptotic bounds for the derivatives of the semigroup in  the directions where smoothing occurs. In this paper, we study the    \emph{large} time asymptotics for the gradients of the diffusion semigroup in the same set of directions and under the same UFG\ condition.       In particular, we identify conditions under which the derivatives of the diffusion semigroup in the smoothing directions decay exponentially in time.  This paper constitutes therefore a stepping stone in the analysis of the long time behaviour of diffusions which do not satisfy the H\"ormander condition. 
\vspace{5pt}
\vspace{5pt}
\\
{\sc AMS Classification (MSC 2010).} 60H10, 35K10,  35B35, 35B65

\end{abstract}

\date{18 April}

\maketitle

\section{Introduction}
Consider the stochastic differential equation (SDE) in $\R^N$
\be\label{SDE}
X_t=X_0 + \int_0^t V_0(X_s) ds + \sqrt{2} \sum_{i=1}^d\int_0^t  V_i(X_s) \circ  dW^i (s),
\ee
where $V_0,\ldots,V_d$ are smooth  vector fields on $\R^N$, $\circ$ denotes Stratonovich integration and, for each $i$,  $W^i(t)$ is an $N$-dimensional standard Brownian motion.  
The Markov semigroup $\{\cP_t\}_{t\ge 0}$ associated with the SDE \eqref{SDE} is defined on the set $C_b$ of continuous and bounded functions, as
\be\label{semigroup}
(\cP_t f)(x):= \mathbb{E} \left[f(X_t \vert X_0=x) \right].
\ee
We recall that, given a vector field $V$ on $\R^N$, we can think of $V$ both as a vector-valued function on $\R^N$ and as a first order differential operator on $\R^N$:
\be\label{notvectoper}
V=(V^1(x), V^2(x) \dd V^N(x)) \quad \mbox{ or }
\quad V= \sum_{j=1}^NV^j(x)\pa_j, 
 \quad x \in \R^N, \pa_j= \pa_{x^j},
\ee
and we shall do so throughout the paper. With this notation, 
the Kolmogorov operator associated with the semigroup $\cP_t$ is the second order differential operator given  on smooth functions by 
\be\label{gen1}
\cL=V_0+\sum_{i=1}^{d}V_i^2. 
\ee
The study of Markov semigroups associated with SDEs of the form \eqref{SDE} has a long history and the literature on the matter is vast. Most of such literature deals with the case in which the operator $\cL$ is elliptic or hypoelliptic; more specifically, a large body of work has been dedicated to the study of the diffusion semigroup \eqref{semigroup} in the case in which the vector fields $V_0,\ldots,V_d$ satisfy the  {\em H\"{o}rmander condition}, in one of its many forms. As is well known, under the (parabolic) { H\"{o}rmander condition}, the transition probabilities of the semigroup $\cP_t$ have a smooth density; furthermore,  $\cP_tf$ is differentiable in every direction and $u(t,x):=(\cP_t f)(x)$ is a classical solution of the Cauchy problem
\begin{align*}
 \pa_t u(t,x) &=\cL u(t,x)\\
 u(0,x)& =f(x).
\end{align*}

In the present  paper we will  relax the hypoellipticity  assumption and work in the setting in which the vector fields $V_0,\ldots,V_d$ satisfy a weaker condition, the so-called UFG condition. The acronym UFG stands for {\em Uniformly Finitely Generated}.  Informally, denoting by $C_b^\infty({\mathbb R}^N)$ the set of smooth bounded functions with bounded derivatives,   this condition states that the $C_b^\infty({\mathbb R}^N)$-module $\mathcal W$ generated by the vector fields $\{V_i,i=1,...,d\}$ within the Lie algebra generated by  $\{V_i,i=0,1,...,d\}$ is finite dimensional. In particular, we emphasize that the UFG condition does not require that the vector space $\{W(x)|W\in \mathcal{W}\}$ is homeomorphic to ${\mathbb R}^N$  for any $x\in {\mathbb R}^N$; indeed,  the dimension of the
space $\{W(x)|W\in \mathcal{W}\}$ is not even required to be constant 
over $\mathbb{R}^N$.  Hence, in this sense, the UFG condition is   weaker than the H\"{o}rmander condition. We give a precise (and easier to check) statement of the UFG condition in Section \ref{sec2}, see Definition 2.1.

In a series of papers \cite{{KusStr82},{KusStr85},{KusStr87}, Kus03}, Kusuoka and Stroock have analyzed the smoothness properties of diffusion semigroups $\{\cP_t\}_{t\ge 0}$ associated with the stochastic dynamics \eqref{SDE} when the vector fields $\{V_i,i=0,1,...,d\}$ satisfy the UFG condition.  In particular they showed that, under the UFG condition, the semigroup $\cP_t$ is no longer differentiable in the direction $V_0$; however it is still differentiable in the direction $\mathcal{V}:=\pa_t-V_0$ and therefore a rigourous PDE analysis  can still be built starting from the stochastic dynamics \eqref{SDE}. In this case one can indeed prove that for  every $f\in C_b$, the function $u(t,x):=(\cP_t f)(x)$ is a classical solution of the Cauchy problem
\begin{equation}
\left\{ 
\begin{array}{rl}
\mathcal{V} u(t,x) &=\sum_{i=1}^{d}V_i^2 u(t,x)\\
 u(0,x)& =f(x).
 \end{array}
\right.\label{linearPDE}
\end{equation}
More precisely, $u$ is twice continuously differentiable  in the directions of the vector fields $V_i$, $i=1,...,d$ and once continuously differentiable in the direction ${\mathcal V}_0=\partial_t-V_0$, when viewed as a function 
$(t,x)\mapsto u(t,x)$ over the  product space 
$(0,\infty)\times {\mathbb R}^d$ (the notion of classical solution for the PDE \eqref{linearPDE} and further background material are gathered in the Appendix).

This fundamental result was obtained by using   probabilistic methods based on the use of the Malliavin calculus (see \cite{Crisan,Nee}). The small time asymptotics  of  $(\cP_t f)(x)$ constitutes the theoretical backbone  for the development of a new class of algorithms, termed {\em cubature methods}, introduced by  Kusuoka, Lyons, Ninomiya and Victoir in the last ten years \cite{cub1, cub2, cub3}. Such algorithms, which work under the UFG condition, provide high order approximations of the law of the solutions of SDEs (and therefore can be used to compute statistical quantities of interest) and are faster than their classical counterparts, see \cite{Crisan}. The study of UFG-diffusions has therefore opened interesting and promising research avenues both in the field of PDE theory and in the field of stochastic simulations.

The papers \cite{{KusStr82},{KusStr85},{KusStr87}} introduce the UFG condition in the context of the theory of diffusion semigroups. However  related conditions had already independently appeared,  in a completely different setting, in the work of Hermann \cite{Hermann}, Lobry \cite{Lobry} and Sussman \cite{Sussman}. In these works, such a condition was considered for control theoretical purposes. More details on the nature of the UFG condition will be given in Section \ref{sec2}.

 Under the UFG condition,  Kusuoka and Strook  proved sharp estimates on  {\em short-time} behaviour of  the semigroup $\cP_t$. Further work on the subject was carried out in \cite{Nee}, where a wealth of results regarding the short-time asymptotycs are derived. To the best of our knowledge nothing is known so far about the {\em long-time} behaviour of the semigroup under the UFG condition. In this paper, we provide the first step towards understanding the   long-time asymptotics of this class of (possibly) degenerate diffusions; in particular, we obtain {\em pointwise} estimates on the time-behaviour of the (space) derivatives of the function $u(t,x)=(\cP_tf)(x)$.
This is the first result concerning the long-time behaviour of UFG semigroups.
The main result of the paper can be informally stated as follows (see Theorem \ref{lemcond1} for a precise statement) \\[2mm]
{\bf Theorem.} {\it If the vector fields $\{V_i,i=0,1,...,d\}$ satisfy both the UFG condition and some quantitative assumption (the ``obtuse angle condition " \eqref{dilcon}) then,  for any bounded continuous function $f$ (not necessarily smooth), any $t_0\in (0,1)$ and any vector field, $V,$ belonging to the $C_b^\infty({\mathbb R}^N)$-module $\mathcal W$, there exist constants  $c_{t_{0}},\lambda>0$ such that   }
\[
| V\cP_tf(x)| ^2 \leq c_{t_{0}} e^{-\lambda t} \qquad {\mbox{for all }} x\in{\mathbb R}^N \ \ \mbox{and all}\quad t>t_{0}.\\[5mm]
 \]
 We emphasize that the UFG condition alone does not suffice to ensure the exponential decay of the coefficients. For a simple counterexample take the one-dimensional Ornstein-Uhlenbeck process with positive drift constant $a>0$. Then the semigroup is uniformly elliptic (hence it satisfies the UFG condition) but one has $\partial_x \cP_t f=e^{at}\cP_t (\partial_x f)$ (see also Note \ref{note41} on this point).


   From a technical point of view, the methods we use in this paper are analytic; indeed, the  strategy we use  to prove our main result,  Theorem \ref{lemcond1},  is a variation of the classic approach established by Bakry   (see \cite{BE,Bakry}) to deduce exponential decay estimates and is similar to the approach adopted by Dragoni, Kontis and Zegarli\'nski in \cite{DragKonZeg}. We defer to  Note  \ref{letbak}  a more careful comparison with this strand of the literature.  Here we just emphasise the pointwise nature of the above inequality. It is indeed customary to obtain bounds for the derivatives of semigroups in $L^p$ spaces weighted by an appropriate invariant measure. This is not possible here, in absence of an obvious invariant measure to exploit. 

To summarize,  the aim of this paper is twofold: i) first, we move another step forward in the Kusuoka-Stroock programme and we produce results that are applicable to the study of cubature methods; ii) second, we extend the classic semigroup approach of Bakry, which was introduced in the context of elliptic diffusions and then applied to hypoelliptic processes, to semigroups which are more general than hypoelliptic.  In particular, regarding the latter point, the estimates obtained in this paper, together with the mentioned control-theoretical results of Herrmann, Lobry and Sussman (\cite{Hermann, Lobry, Sussman}), will form the stepping stone for future work on the ergodic theory for SDEs with generator which does not necessarily satisfy the H\"ormander condition (Corollary \ref{corcontrol} is a simple example in this spirit).  On a related note, we would like to emphasize that some commonly used diffusion processes do not satisfy the H\"ormander condition, but satisfy the UFG condition; the simplest of such examples is Geometric Brownian motion. Another important motivation for the current work is to provide the basis of the asymptotic (in time) analysis of the error incurred by the high order numerical approximations produced by cubature methods.

The paper is organized as follows: in Section \ref{sec2} we introduce the UFG condition and the necessary notation. In Section \ref{sec3} we present a version of the classical Bakry technique, adapted to our context. In Section \ref{sec4} we present our main results concerning the exponential decay of the derivatives of the semigroup and explain how such estimates can be obtained by employing the techniques presented in Section \ref{sec3}. In Subsection \ref{subs41} we show one way of using our estimates to obtain information on the behaviour of the semigroup itself. More detailed results in this direction will be the object of future work. In Section \ref{sec5} we gather all the proofs of the results of Section \ref{sec4}.  

\section{The UFG condition and notation}\label{sec2}
Fix $d \in \N$ and let $\A$ be the set of all $n$-tuples, of any size $n \geq 1$,  of integers of the following form 
$$
\A:= \{\alpha=(\alpha^1 \dd \alpha^n), n \in \N : \alpha^j\in \{0, 1 \dd d\} {\mbox{ for all } j \geq  1} \}\setminus \{(0)\}\,.
$$
For the sake of clarity, we stress that all $n$-tuples of any length $n\geq 1$ are allowed in $\A$,  except the trivial one, $\alpha=(0)$ (however $\alpha=(j)$ belongs to $\A$ if $j\in\{1\dd d\}$). 
We endow $\A$ with the product 
$$
\alpha \ast \beta:= (\alpha^1 \dd \alpha^h, \beta^1 \dd \beta^{\ell}), 
$$
for any $\alpha=(\alpha^1 \dd \alpha^h)$ and 
$\beta=(\beta^1 \dd \beta^{\ell})$ in $\A$.  If  $\alpha$ is an element of $\A$, we define the {\em length} of $\alpha$, denoted by $\|\alpha\|$, the integer
$$
\|\alpha\|:= h+\mbox{card}\{i: \alpha_i=0\} , \qquad \mbox{if } \alpha=(\alpha^1 \dd \alpha^h) \,.
$$
For any $m \in \N, m\geq 1$, we then  introduce the sets
\begin{align*}
&\A_m=\{\alpha \in \A: \| \alpha\|\leq m   \}
\end{align*}
and if $B$ is any set,  $\lv B\rv$ will denote the cardinality of the set $B$.  \footnote{We hope that this does no create confusion when $x\in \R^N$, in which case $\lv x\rv$ is the euclidean norm of $x$.}

Given a vector field (or, equivalently, a first order differential operator) $V=(V^1(x), V^2(x),$ $..., V^N(x))$  on $\R^N$, we refer to the functions $\{V^j(x)\}_{1 \leq j \leq N}$ as to  the {\em components} or {\em coefficients} of the vector field.   We  say that a vector field on $\R^N$ is smooth or that it is $C^{\infty}$ if all the components $V^j(x)$, $j=1 \dd N$, are $C^{\infty}$ functions.
Given two differential operators $V$ and $W$, the commutator between $V$ and $W$ is defined as
$$
[V,W]:= VW -WV\,.
$$
Let now $\{V_i: i=0 \dd d \}$ be a collection of vector fields on $\R^N$ and let us define the following ``hierarchy" of operators:
\begin{align*}
V_{[i]} &:= V_i \qquad i=0, 1 \dd d\\
V_{[\alpha \ast i]} & := [V_{[\alpha]}, V_{[i]}], \qquad \alpha \in \A, i=0,1 \dd d\,.
\end{align*}
Note that if $\| \alpha\|=h$ then $\| \alpha \ast i \| = h+1$ if $i \in \{1 \dd d\}$ and $\| \alpha \ast i\|= h+2$ if $i=0$. If $\alpha \in \A$ is a multi-index of length $h$, with abuse of nomenclature we will say that $V_{[\alpha]}$ is a  differential operator of length $h$. 
 We can then define the space  $\mathfrak{R}_m$ to be  the space containing all the operators of the above  hierarchy, up to and including the operators of length $m$ (but excluding $V_0$):
\be\label{(R)}
\mathfrak{R}_m:=\left\{ V_{[\alpha]}, \alpha \in \A_m\right\}. 
\ee
Let 
$$
\mathrm{span} \{ \mathfrak{R}_m\}
:= \left\{\mbox{vector fields }V \mbox{ on } \R^N:  V= \sum_{\beta \in \A_m} \varphi_{\alpha,\beta} V_{[\beta]} (x) \right\},
$$
where the functions $\varphi_{\alpha, \beta}$ in the above belong to the set $C^{\infty}_V(\R^N)$ of  bounded smooth functions,  $\varphi_{\alpha, \beta}= \varphi_{\alpha, \beta}(x): \R^N \rightarrow \R$,   such that
\be\label{supphi}
\sup_{x \in \R^N}\lv V_{[\gamma_{(1)}]} \dots V_{[\gamma_{(n)}]} \varphi_{\alpha,\beta} \rv < \infty
\ee
for all $n$ and all $\gamma_{(1)} \dd \gamma_{(n)}, \alpha$ and $\beta$ in $\A_m$. 
With this notation in place we can now introduce the definition that will be central in this paper.

\begin{definition}[UFG Condition]\label{defufg}
Let $\{V_i: i=0 \dd d \}$ be a collection of 
smooth vector fields on $\R^N$ and assume that the coefficients of such vector fields have bounded partial derivatives (of any order). We say that the fields  $\{V_i: i=0 \dd d \}$ satisfy the UFG condition if there exists $m\in\N$ such that for any $\alpha \in \A$ of the form 
$$
\alpha = \alpha'\ast i, \qquad    \alpha' \in \A_m, \, i \in \{0 \dd d\}, $$ 
 there exist  bounded smooth functions $\varphi_{\alpha, \beta} \in C^{\infty}_V(\R^N)$ such that
$$
V_{[\alpha]}(x) = \sum_{\beta \in \A_m} \varphi_{\alpha,\beta} V_{[\beta]} (x) \,.
$$
\end{definition}
We emphasize that the set of vector fields appearing in the linear combination on the right hand side of the above identity, does not include $V_{0}$. 
\begin{example}\textup{
Consider the following first order differential operators on $\R^2$
$$
V_0=\sin x \,\partial_y \qquad V_1=\sin x\,\partial_x \,.
$$
Then $\{V_0,V_1\}$ do not satisfy the H\"ormander condition (e.g. there is always a degeneracy at $x=0$) but they do satisfy the UFG condition with $m=4$. If the role of the fields is exchanged, i.e. if we set 
$$
V_0=\sin x\,\partial_x ,  \qquad V_1=  \sin x \,\partial_y\,
$$
then $\{V_0, V_1\}$ still satisfy the UFG condition, this time with $m=1$ (indeed,  $[V_0,V_1]=\cos x V_1$).
{\hfill $\Box$}}
\end{example}

\begin{note}\textup{ Under the assumption \eqref{supphi} on the functions $\varphi$,   if the UFG condition holds for some $m \in \N$ then it also holds for any $n\geq m, n \in \N$.  In other words, if the UFG condition holds for some $m$ in $\N$ then for any $V_{[\gamma]}$ with $\|\gamma\|> m$ one has 
$$
V_{[\gamma]}= \sum_{\beta \in \A_m} \varphi_{\gamma,\beta} V_{[\beta]} (x) 
$$
for some  bounded  functions $\varphi_{\gamma,\beta}$. 
For this reason it is appropriate to remark that in the remainder of the paper, when we assume that ``the UFG condition is satisfied for some $m$", we mean the smallest such $m$.  
}{\hfill$\Box$} 
\end{note}

In this paper we will  consider diffusion semigroups $\{\cP_t\}_{t\geq 0}$ of the form \eqref{semigroup}; that is, we consider Markov semigroups associated with the stochastic dynamics \eqref{SDE}. In particular, we will be interested in studying the semigroup $\cP_t$ when the vector fields $\{V_0, V_1 \dd V_d\}$ satisfy the UFG condition. We recall that a  semigroup $\cP_t$  of bounded operators  is Markov if 
$$
\cP_t 1= 1 \qquad \mbox{and} \qquad \cP_t f \geq 0 \mbox{ when } f\geq 0\,,
$$ 
where, in the above, $1$ denotes the function identically equal to one. Denoting by $\| \cdot \|_{\infty}$ the supremum norm, the above implies  that if $\|f\|_{\infty}< \infty$ then $\|\cP_t f\|_{\infty} \leq \|f\|_{\infty}$, i.e. the semigroup is a contraction in the supremum norm.

The UFG condition is strictly weaker than the uniform H\"ormander condition (see \cite{{CrisanGhazali}}). However one can still prove that, when such  a condition is satisfied by the vector fields $\{V_0, V_1\dd  V_d\}$ appearing in the generator \eqref{gen1}, the semigroup $\cP_t$ still enjoys good smoothing properties: if $f(x)$ is continuous then $(\cP_t f)(x)$ is differentiable in all the directions spanned by the vector fields contained in $\mathfrak{R}_m$ (we recall that the set $\mathfrak{R}_m$ is defined in \eqref{(R)}).
\footnote{Actually, differentiability holds in all the directions spanned by the vector fields $V_{[\alpha]}, \, \alpha \in \A $. Notice that differentiability in the direction $V_0$ does not in general hold under the UFG condition. This is one of the major differences with the uniform H\"ormander condition, see \cite[Section 2.9]{Crisan}.} 

Moreover, whilst the function  $u(t,x):= (\cP_t f)(x)$ may not be differentiable in the direction  $V_0$ or in the time variable, it is   still differentiable in the direction $\mathcal{V}:=\pa_t-V_0$, and it is the unique classical solution of the Cauchy problem
\begin{align}
\mathcal{V} u(t,x)&=  \sum_{j=1}^d V_j^2 u(t,x) \label{E}\\
u(0,x)&=f (x), \nonumber
\end{align}
provided the initial datum $f$ is continuous and bounded. For the reader's convenience we include in Appendix the definition of classical solution for the PDE \eqref{E}. 

Suppose now, and for the remainder of this section,  
that the operators $\{V_0, V_1 \dd V_d\}$ satisfy the UFG condition for some $m>0$. We can then construct the vector field $\V$, containing all the vector fields (operators) $V_{[\alpha]}$,  $\alpha \in \A_m$:
\be\label{V}
\V:=\left(  V_1 \dd V_d,  \dots, V_{[\alpha]}, {\dots}  \right)\,.
\ee
The vector $\V$ has $\lv \A_m\rv$ entries; using the notation \eqref{notvectoper}, each entry (i.e. each vector $V_{[\alpha]}$,  $\alpha \in \A_m$) can be expressed as follows:
$$
V_{[\alpha]}= (V_{[\alpha]}^1 \dd V_{[\alpha]}^N).
$$
Therefore $\V$ can be rewritten as 
$$
\V=\left(  V_1^1\dd V_1^N \dd V_d^1 \dd V_d^N,  \dots, V_{[\alpha]}^1, {\dots}, V_{[\alpha]}^N, {\dots}  \right). 
$$
It is clear from the above that we can think of $\V$ as a function from $\R^N$ to $\R^{N \lv \A_m\rv}$. However we will most often think of $\V$ as a vector of operators rather than as a vector of vectors  and therefore we will adopt the notation \eqref{V}. 
More in general, the space of vectors  with $\lv \A_m\rv$ entries, where each entry is an operator in  span$\{\mathfrak{R}_m\}$, will be denoted by $\RR$. Clearly, $\V \in \RR$.

We emphasize  that if  $\X \in \RR$ , then $\X$ will always be denoted in bold font while the component  of $\X$ corresponding to the multi-index $\alpha$ is simply a differential operator and it is therefore  denoted by $X_{[\alpha]}$.  If $V_j$ is any first order differential operator,  we also write
$$
V_j\V= \left( V_j V_1 \dd V_j V_d, {\dots}, V_jV_{[\alpha]}, {\dots}\right)\,.
$$
Given a collection of strictly  positive numbers $\{a_{[\alpha]}\}_{\alpha \in \A_m}$ and  any $f(x):\R^N \rightarrow \R$ (smooth enough so that the expression  below makes sense), we can define the following quadratic form:
\be\label{Lyapfun}
(\Gamma f)(x):= \sum_{\alpha \in \A_m} 
a_{[\alpha]} \lv (V_{[\alpha]}f)(x)\rv^2\,, \qquad x \in \R^N.
\ee
If a multi-index $\alpha$ is of length $k$, we will denote it by $\alpha_k$ (when we want to emphasize   its length) and $V_{[\alpha_k]}$ will be the corresponding  first order operator of length $k$ (obviously, for a given $k\in \N$, there are many multi-indices of length $k$ and, correspondingly, many operators of length $k$). With this more detailed notation, the quadratic form $\Gamma$ can equivalently be expressed as
$$
(\Gamma f) (x)= \sum_{k=1}^m \sum_{\{\alpha_k ,\\ \| \alpha_k\|=k\}} \al{k} \lv \vl{k}(x) \rv^2\,.
$$
Also, if we define the following bilinear form on $\RR$
\be\label{scpr}
\la \X f, \Y f \ra\am:=\sum_{k=1}^m \sum_{\alpha_k :\\ \| \alpha_k\|=k} \al{k} (X_{[\alpha_k]}f) (Y_{[\alpha_k]}f),   \qquad \X, \Y \in \mathfrak{R}^{\lv \A_m \rv},   
\ee
where $f$ is any smooth enough function,
then the quadratic form $\Gamma$ can be rewritten as 
\be\label{equivnot}
\Gamma(f)= \| \V f\|\am^2\,,
\ee
where $\| \cdot \|\am$ is the (semi) norm induced by the bilinear form $\la \cdot, \cdot \ra\am$. 
We stress that the definition of the bilinear form $\la \cdot, \cdot \ra\am$ depends on the choice of the constants $\{a_{[\alpha]}\}_{\alpha \in \A_m}$. For $j \geq 0$ we also define the  linear mappings $\Lambda_j, \Lambda: \mathrm{span}\left\{ \mathfrak{R}_m \right\}\rightarrow \mathrm{span}\left\{ \mathfrak{R}_m \right\}$ as follows:
\be\label{lambdaj}
\Lambda_j V_{[\alpha]} = \left\{ 
\begin{array}{ll}
V_{[\alpha \ast j]} & \mbox{ if } \| \alpha \ast j \| \leq m\\
& \\
\sum_{\beta \in \A_m} \varphi_{\alpha \ast j, \beta } V_{[\beta]} & \mbox{ if } \| \alpha \ast j \| > m \, ,
\end{array}
\right.
\ee
and 
\be\label{lambda}
\Lambda:= \Lambda_0+ \sum_{j=1}^d \Lambda_j \Lambda_j \,.
\ee
With abuse of notation, we keep denoting by $\Lambda_j$ also the linear mapping $\Lambda_j: \RR \rightarrow \RR$ that acts on the component $[\alpha]$ of the vector $\V$ as follows:
$$
(\Lambda_j \V)_{[\alpha]}:= \Lambda_jV_{[\alpha]}.
$$
Analogous  use of notation holds for $\Lambda$ as well.
\begin{note}\label{note:indufg}\textup{
In view of Proposition \ref{Bakrystrategy} below, we remark that  all the objects defined so far, in particular the quadratic form $\Gamma$ and the maps $\Lambda$ and $\Lambda_j$, make sense, at least formally,  irrespective of whether the UFG condition holds. In other words, the integer $m$ appearing in the definitions of such objects could be any integer. Obviously, when the UFG condition holds with $m$, then all such definitions become meaningful for our purposes.  
}\hf
\end{note}
If the UFG condition holds with $m$, then we will denote by  $Pol$ the set of functions $f$ which are differentiable  in the directions $V_{[\alpha]}, \alpha \in \A_m,$ (but not necessarily in other directions) and such that 
\be\label{polcon}
\lv (V_{[\alpha]} f) (x)\rv^2 \leq \kappa (1+\lv x \rv^q),
\ee
for some $\kappa, q>0$. When, given a function $f\in Pol$, we want to stress the value of the constant $\kappa$ such that the above holds, we write $f \in Pol({\kappa})$.

We conclude this section by gathering some preliminary basic facts that we will  repeatedly  use  in the remainder of the paper and by presenting a simple example to illustrate the notation introduced so far. 
\begin{itemize}
\item If  $X,Y$ and $Z$ are  any three  first order differential operators then 
$$
[X, YZ] = [X,Y] Z + Y [X,Z]\,.
$$
\item If $\cL$ is the operator \eqref{gen1},  using the above we find that for any vector field $V_{[\alpha]}$:
\begin{align}
[V_{[\alpha]}, \cL]&= V_{[\alpha \ast 0]} 
 + \sum_{j=1}^d V_{[\alpha \ast j \ast j]}+ 2 \sum_{j=1}^d V_j V_{[\alpha \ast j]}  \label{vacoml}\\
&= \Lambda V_{[\alpha]}+  2 \sum_{j=1}^d V_j \Lambda_j V_{[\alpha]} \,. \nonumber
\end{align}
\end{itemize}

 \begin{example}[UFG-Heisenberg  Lie algebra] \label{exa:Heisenberg}\textup{We call this example the UFG-Heisenberg Lie algebra, as it is obtained by a modification of the so-called Heisenberg Lie algebra (which is the Lie algebra of vector fields that are invariant with respect to the action of the Heisenberg group on $\R^3$, see \cite{BLU}). More precisely, set $d=2$ and $N=3$  and consider the operators
\begin{align*}
 & X_1:=\pa_x-\frac{y}{2}\pa_z, \quad X_2:=\pa_y+\frac{x}{2}\pa_z,  \quad X_3:= [X_1, X_2]:= \pa_z, \\
& X_0:=xX_1+yX_2+2zX_3= x\pa_x + y \pa_y + 2z\pa_z \,.
\end{align*}
The Lie algebra generated by  $\{X_0, X_1, X_2\}$ is usually referred to as the Heisenberg Lie algebra. The vector fields $\{X_0, X_1, X_2\}$ satisfy the 
H\"ormander condition hence the operator $\cL=X_0+X_1^2+X_2^2$ is hypoelliptic on $\R^3$. If the above fields are slightly modified, we obtain new vector fields, $\{V_0, V_1, V_2\}$,  that no longer satisfy the H\"ormander condition, but satisfy the UFG condition instead. Indeed, 
 let again $d=2$ and $N=3$ and consider the operators
$$
V_0 := -k (x\pa_x+y \pa_y +2 z\pa_z), \quad V_1:= -y\pa_z, \quad V_2:=\pa_y+x \pa_z, \qquad k>0.
$$
The operators $\{V_0, V_1, V_2\}$ satisfy the UFG condition with $m=2$, as
\begin{align*}
& [V_1,V_0]=-kV_1,  \,\,\quad\qquad  [V_2, V_0]=-k V_2\\
& [V_1, V_2]=\pa_z = V_{12}, \qquad [V_{12},V_0]= -2kV_{12}, 
\qquad  [V_{12}, V_1]=[V_{12}, V_2]=0\,.
\end{align*}
 Therefore, in this example we have 
$\A_2:=\{1,2,  (1,2), (2,1)\}$ and $\textup{span}\{\mathfrak{R}_2\}=\textup{span}\{V_1, V_2, V_{[1\ast 2]}=:V_{12}\}$.  Because $V_{21}:= V_{[2 \ast 1]}=-V_{12}$, $V_{21}$ doesn't need to be in the list of the base fields of $\mathfrak{R}_2$ (for the same reason it can also be omitted in the definition of $\Gamma$ below, as the constants $a_1, a_2, a_{12}$ are anyway arbitrary).  Using the definition \eqref{Lyapfun},  the quadratic form $\Gamma$ associated with the UFG-Heisenberg group is 
$$
(\Gamma f) (x)= a_1  \lv V_1 f_t\rv^2+
a_2 \lv V_2 f_t\rv^2+ a_{12} \lv V_{12} f_t\rv^2.
$$
The vector $\V$ is 
$
\V=(V_1, V_2, V_{12})
$
 and the mappings $\Lambda_1$ and $\Lambda_2$ give
$$
\Lambda_1 \V = (0, - V_{12}, 0), \qquad 
\Lambda_2 \V = (V_{12}, 0, 0)\,,
$$
while for $\Lambda_0$ we have
$$
\Lambda_0 \V=(-kV_1, -kV_2, -kV_3).
$$
}
{\hfill$\Box$} 
\end{example}


\section{Preliminary results: a Bakry-type approach}\label{sec3}

In this section we consider Markov semigroups associated with operators $\cL$ of the form \eqref{gen1}, for a given set $\{V_0, V_1 \dd V_d\}$ of vector fields.  We recall that the class of functions $Pol$ has been defined immediately after Note \ref{note:indufg}. 

\begin{prop}\label{Bakrystrategy}
Let $f_t:=\cP_t f_0$ be the  diffusion semigroup defined in \eqref{semigroup}. 

\smallskip

\textbf{\textup{(a)}} Let $m$ be any positive integer and    assume the initial datum $f_0$ is a bounded smooth (in every direction) function such that $\| V_{[\alpha]} f_0\|_{\infty}< \infty$ for all $\alpha \in \A_m$.  Consider the quadratic form  $\Gamma$ defined in \eqref{Lyapfun}:
$$
(\Gamma f_t)(x):= \sum_{\alpha \in \A_m} a_{[\alpha]} \lv V_{[\alpha]} f_t(x) \rv^2 \,,
$$
for some strictly positive constants $\{a_{[\alpha]}\}_{\{\alpha \in \A_m\}}$ (to be chosen).  Suppose there exists  $\lambda>0$ such that the following inequality holds:
\be\label{glcon}
\frac{d}{ds} \cP_{t-s} \Gamma(f_s(x)) \leq - \lambda \cP_{t-s} \Gamma(f_s(x))\, \quad \mbox{ for any } x \in \R^N.
\ee
Then 
\be\label{expdec}
\Gamma(f_t) \leq e^{-\lambda t} \| \Gamma (f_0)\|_{\infty},  \quad \mbox{for all } t\ge0; 
\ee
therefore, 
$$
\lv V_{[\alpha]}f_t (x)\rv ^2 \leq \frac{1}{a_{[\alpha]}} {\| \Gamma (f_0)\|_{\infty}} e^{-\lambda t} \qquad {\mbox{for all }} \alpha \in \A_m, t\ge 0 \,.
$$
\smallskip

\textbf{\textup{(b)}}Suppose, in addition, that  the vector fields  $\{V_0 \dd V_d\}$ satisfy the UFG  condition (for some $m$).  In this case, if \eqref{glcon} holds when $f_0$ is smooth (and $\| V_{[\alpha]} f_0\|_{\infty}< \infty$), then the following holds for any $f_0\in Pol(\kappa)$:  for every open ball $\mathbb{B}(0,K)$ of radius $K$ and for all $\alpha \in \A_m$,
\be\label{local1}
\sup_{x \in \mathbb{B}(0,K)}\lv V_{[\alpha]}f_t (x)\rv ^2 \leq  \kappa \, c_K \,e^{-\lambda t}, 
\ee
 where $c_K>0$ is a constant dependent on $K$.

\smallskip

\textbf{\textup{(c)}} If  the vector fields  $\{V_0 \dd V_d\}$ satisfy the UFG  condition (for some $m$) and \eqref{glcon} is satisfied for any smooth initial datum, then then following holds when  $f_0$ is only continuous and bounded  (but not necessarily smooth):  for any $t_0\in (0,1)$ and any $K>0$
 there exists a constant $c_{t_0,K}>0$ such that   
\be\label{expdec3}
\sup_{x \in \mathbb{B}(0,K)}\lv V_{[\alpha]}f_t (x)\rv ^2 \leq c_{t_{0},K} \,  e^{-\lambda (t-t_0)}
\|f_0(x)\|_{\infty}^2 \qquad {\mbox{for all }} \alpha \in \A_m \ \mbox{and all}\quad t>t_{0}. 
\ee
Moreover, if the coefficients of the vector fields $\{V_0 \dd V_d\}$ are bounded, then the constant $c_{t_0,K}$ does not depend on $K$ and we have the uniform bound 
\be\label{expdec4}
\lv V_{[\alpha]}f_t (x)\rv ^2 \leq c_{t_{0}} \,  e^{-\lambda (t-t_0)} \|f_0(x)\|_{\infty}^2 \qquad {\mbox{for all }} \alpha \in \A_m \ \mbox{and all}\quad t>t_{0}. 
\ee
\end{prop}
Before proving the above result we make the following remark, which we will use in the proof of Proposition \ref{Bakrystrategy}. We will make several comments on the above statement in Note \ref{letbak}. 
\begin{note}\label{noteregularity}\textup{If the initial datum $f_0$ is bounded and continuous and the UFG condition holds,  then $(\cP_t f_0)$ is differentiable in the directions $V_{[\alpha]}, \alpha \in \A_m $ (see \cite{KusStr82}-\cite{Kus03}). Because we are assuming that the vector fields  $\{V_0 \dd V_d\}$ are smooth, the semigroup is differentiable an arbitrary number of times  in such directions.   Moreover the following short time asymptotic holds: for any ball of radius $K$, $\mathbb{B}(0,K)$,  and for any $\alpha \in \A_m$, 
\be\label{smoothingdet}
 \sup_{x \in \mathbb{B}(0,K)} \lv V_{[\alpha]} (\cP_t f_0)(x) \rv \leq  \frac{\mathfrak{c}}{t^{\| \alpha \|/2}}
\sup_{x \in \mathbb{B}(0,K)} \left[ \left( 1+\lv x\rv^{\| \alpha\|}\right)  \right] \|f_0(x)\|_{\infty} ,
\ee
for some constant $\mathfrak{c}>0$ (which does not depend on $x,t$ or $f_0$). 
Details about the above short-time asymptotics (and many other  results of this type) can be found  in \cite{Nee} (see in particular \cite[pages 68-80]{Nee}). Furthermore,  when the vector fields $V_{[\alpha]}$ have bounded coefficients, the following holds:
\be\label{smoothing1}
\lv V_{[\alpha]} (\cP_t f_0)(x) \rv \leq \frac{\tilde{c}}{t^{\| \alpha \|/2}} \,.
\ee
}
\end{note}
\begin{proof}[Proof of Proposition \ref{Bakrystrategy}]
\textbf{\textup{(a)}} This is completely standard: By applying Gronwall's lemma, from  \eqref{glcon} we deduce 
\be\label{prov}
\cP_{t-s} \Gamma(f_s) \leq e^{-\lambda s} \,\cP_t \Gamma(f_0), 
\quad \mbox{for all } \, 0< s\leq  t \,.
\ee
Therefore, using   \eqref{prov} for $s=t$ and the contractivity  of the semigroup  $\cP_t$ in the supremum norm gives the result. Notice in particular that if $f_0$ is smooth in every direction then also $(\cP_tf_0)(x)$ is smooth in every direction (see Appendix); in particular, it is smooth in $t$ as well, hence all of the above is justified.\\
\textbf{\textup{(b)}} We prove this statement in the Appendix, see Lemma \ref{densitylemma}. \\
\textbf{\textup{(c)}} Using Note \ref{noteregularity}, notice that for any $t_0\in (0,1)$ the function $\cP_t f_0$ belongs to the set $Pol(\kappa)$; in particular, by \eqref{smoothingdet}, the constant $\kappa$ appearing in \eqref{polcon} is, for this function,  $\kappa= t_0^{-m} \mathfrak{c} \|f_0(x)\|_{\infty}$. Therefore, by   part (b),   for any fixed  $0<t_{0} <1$ and for any $t\ge t_0$,  we can write
\begin{align*}
\lv  V_{[\alpha]} (\cP_t f_0 (x))\rv^2 &=  
\lv  V_{[\alpha]} (\cP_{t-t_{0}} \cP_{t_0} f_0 (x))\rv^2 \leq e^{-\lambda(t-t_0)}c_{t_0,K}
\|f_0(x)\|_{\infty}.
\end{align*} 
If the coefficients of the $V_{[\alpha]}$'s  are bounded, then \eqref{smoothing1} gives \eqref{expdec4} by acting analogously to what we have just done.
 \end{proof}

\begin{note}\label{letbak}\textup{
 Proposition \ref{Bakrystrategy} part (a) provides a general framework  to deduce the exponential decay for the derivatives of diffusion semigroups;  part (a)  is just the classic  Bakry approach\cite{BE, Bakry},  readapted to our purposes. In particular:}
\begin{itemize}
\item \textup{Proposition \ref{Bakrystrategy} part (a) is {\em not} a smoothing result, it is  just a long time asymptotics. Indeed
in the statement of part (a)  we assumed that the initial datum $f_0(x)$ is a smooth function with bounded derivatives. This is to  make sense of the expression $(\Gamma f_0) (x)$ and to be able to take time-derivatives in \eqref{glcon}. Such a result is quite general and it is  independent of whether the UFG condition holds (see also Note \ref{note:indufg} in this respect). }
\item \textup{Once the exponential decay \eqref{expdec} is obtained for smooth initial data, one can use the semigroup property and the smoothing effects which are guaranteed to hold under the UFG condition (and quantified by the estimates \eqref{smoothingdet}- \eqref{smoothing1}) in order to prove exponential decay of the derivatives of the semigroup for any initial datum $f_0(x)$ which is just continuous and bounded.  This is the content of part ((b) and) (c) of Proposition  \ref{Bakrystrategy}. Therefore, in the proof of our main results we just need to focus on showing exponential decay for smooth initial data. 
}
\item \textup{The analysis used here is based on the adaption of the Bakry technique used in \cite{DragKonZeg}. The difference between the quadratic forms $\Gamma$ that we use here and  those  considered in \cite{{DragKonZeg}} is the appearance of the constants $\al{k}$. That is, the quadratic form used in \cite{{DragKonZeg}}  
 can be  obtained from ours by just setting $\al{k}=1$ for all $\al{k}$. Introducing the positive parameters $\al{k}$, which can be conveniently chosen,  allows us to have a better estimate for the decay rate $\lambda$ (see Note \ref{pc}). To the best of our knowledge, the idea of introducing such parameters first appeared  in \cite{{Herau2007}} and was then further developed in \cite{V}. However \cite{{Herau2007},V} work in weighted spaces, the weight being the invariant measure of the semigroup. Here there is no obvious invariant measure to exploit, hence we have to work in a pointwise setting, similar to \cite{DragKonZeg}.}
\end{itemize}
\hf
\end{note}

The result of Proposition \ref{Bakrystrategy} part (a) hinges only on proving \eqref{glcon}.  The following elementary lemma gives a sufficient condition to verify \eqref{glcon}.   
 Before stating the next lemma we observe that, with our assumptions on the coefficients of the SDE \eqref{SDE},  classic arguments show that the operator $\cL$ and the semigroup commute on a set of sufficiently smooth functions (say e.g. on the set $C_V^{\infty}$, defined just before Definition \ref{defufg}).

\begin{prop}\label{Bakrystrategy2} Assume the same setting of Proposition \ref{Bakrystrategy} part (a). 
 If there exists a real number $\lambda>0$ such that
\be\label{gcon}
\left( -\cL + \pa_t\right) \Gamma(f_t) \leq  - \lambda \Gamma(f_t)\quad 
\forall t>0
\ee
then \eqref{glcon}  holds. 
\end{prop}
\begin{proof}
This is again standard so we only sketch it. 
$$
\frac{d}{ds} \cP_{t-s} \Gamma(f_s(x))= - \cL\cP_{t-s} \Gamma(f_s)+
\cP_{t-s}\pa_s \Gamma(f_s)\,.
$$
We can now use  the fact that the semigroup commutes with its generator (on a set of sufficiently smooth functions) and the positivity preserving property of Markov semigroups, and therefore conclude the proof. 
\end{proof}

\section{Main Results: Long time behaviour of derivatives  of the semigroup} \label{sec4}
 If $X$ is a first order differential operator on $\R^N$,  $\cL$ is the operator \eqref{gen1} and $f_t(x):= (\cP_t f_0)(x)$ then 
\be\label{lem1}
(-\cL + \pa_t) \lv X f_t \rv^2 = -2 \sum_{j=1}^d \lv V_j X f_t\rv^2 + 2 
\left( [X, \cL] f_t \right) (X f_t) , 
\ee
whenever $f_0$ is smooth. The identity \eqref{lem1} is obtained by  using \eqref{E} and the fact that $X$ and all the $V_j$'s are first order differential operators (see \cite[Lemma 2.2]{MV_I}).
Recall that if a multi-index $\alpha$ is of length $k$ we will denote it by $\alpha_k$. 
In view of \eqref{gcon}, we use \eqref{lem1} to  calculate the following  
\begin{align*}
(-\cL + \pa_t) \Gamma (f_t) & \stackrel{\eqref{lem1}}{=} -2 \sum_{k=1}^m \sum_{\alpha_k} \al{k}\sum_{j=1}^d
\lv V_j \vl{k} \rv^2 + 2 \sum_{k=1}^m \sum_{\alpha_k}  \al{k} 
\left([V_{[\alpha_k]}, \cL] f_t \right) \left( \vl{k}\right)\\
& \stackrel{\eqref{vacoml}}{=} -2 \sum_{k=1}^m \sum_{\alpha_k} \al{k}\sum_{j=1}^d
\lv V_j \vl{k} \rv^2 
+ 4 \sum_{k=1}^m \sum_{\alpha_k} \al{k}\sum_{j=1}^d 
\left( V_j V_{[\alpha_k \ast j]} f_t\right) \left( \vl{k} \right)\\
&+2 \sum_{k=1}^m \sum_{\alpha_k}  \al{k}
\left([V_{[\alpha_k]}, V_0] f_t \right) \left( \vl{k}\right)\\
&+ 2 \sum_{k=1}^m \sum_{\alpha_k} \al{k}
\sum_{j=1}^d 
\left( V_{[\alpha_k \ast j \ast j]} f_t\right) \left( \vl{k} \right)\,.
\end{align*}

Notice that if $V_{[\alpha_k]}$ is a field of length $k$,  from  \eqref{vacoml} one can see that the commutator between 
$V_{[\alpha_k]}$ and $\cL$ will contain second order operators summed up with  first order operators of length at most $k+2$. For this reason, when we calculate the commutators $[V_{\alpha_{m}}, \cL]$ and $[V_{\alpha_{m-1}}, \cL]$, we can make use of the UFG condition and express such commutators in terms of fields of length at most $m$. 
 This fact will be repeatedly used in the proofs of Section \ref{sec5}. 

We now split the above expression as follows:

\be\label{user}
(-\cL + \pa_t) \Gamma (f_t)= \mathcal{S}(f_t)+ \mathcal{F}(f_t)
\ee
where
\begin{align}
\mathcal{S}(f_t) &:=
-2 \sum_{k=1}^m \sum_{\alpha_k} \al{k}\sum_{j=1}^d
\lv V_j \vl{k} \rv^2 
+ 4 \sum_{k=1}^m \sum_{\alpha_k} \al{k}\sum_{j=1}^d 
\left( V_j V_{[\alpha_k \ast j]} f_t\right) \left( \vl{k} \right) \label{Scom}\\
& \stackrel{\eqref{lambdaj}}{=}-2 \sum_{j=1}^d \| V_j \V f_t\|\am^2+ 4 \sum_{j=1}^d \la V_j \Lambda_j \V f_t, \V f_t  \ra\am \label{S}
\end{align}
and 
\begin{align}
\mathcal{F}(f_t) &:=
+2 \sum_{k=1}^m \sum_{\alpha_k}  \al{k}
\left([V_{[\alpha_k]}, V_0] f_t \right) \left( \vl{k}\right)\nonumber\\
&+ 2 \sum_{k=1}^m \sum_{\alpha_k} \al{k}
\sum_{j=1}^d 
\left( V_{[\alpha_k \ast j \ast j]} f_t\right) \left( \vl{k} \right) 
\nonumber 
\\
& \stackrel{\eqref{lambda}}{=} 2 \la \Lambda \V f_t, \V f_t \ra\am\,.\label{F}
\end{align}
Notice that $\mathcal{F}(f_t)$ contains only first order operators (vector fields), while $\mathcal{S}(f_t)$ contains second order as well as first order operators (see also the expression for $\mathcal{S}(f_t)$ at the beginning of the proof of Lemma \ref{lemcond1}, in particular the terms with ($\ast \ast$)).

\begin{theorem}\label{lemgenericassp} Let $m$ be a positive integer and $\cP_t f_0=:f_t$ be the semigroup associated with the SDE \eqref{SDE}, i.e. the semigroup \eqref{semigroup}.  With the notation introduced so far, 
assume  the following two conditions are satisfied by the vector fields $\{V_0, V_1 \dd V_d\}$ appearing in \eqref{SDE}:
\begin{itemize}
\item there exists a collection of strictly positive constants $\{a_{[\alpha]}\}_{\alpha \in \A_m}$ such that the corresponding bilinear form \eqref{scpr} satisfies  
\be\label{cond1}
\mathcal{S}(f_t)=-2 \sum_{j=1}^d \| V_j \V f_t\|\am^2+ 4 \sum_{j=1}^d \la V_j \Lambda_j \V f_t, \V f_t \ra\am
\leq \gamma \|  \V f_t\|\am^2,
\ee
for some constant $\gamma>0$ (possibly dependent on the collection $\{a_{[\alpha]}\}_{\alpha \in \A_m}$);

\item there exists $\mu> \gamma$ such that
\be\label{cond2}
\mathcal{F}(f_t)= 2 \la \Lambda \V f_t, \V f_t \ra\am \leq -  \mu \|  \V f_t\|^2\am \, ,
\ee
where $\la \cdot , \cdot \ra\am$ is the bilinear form defined by the same constants for which \eqref{cond1} holds. 
\end{itemize}
Then \eqref{gcon} holds with $\lambda= \mu- \gamma$. Therefore \eqref{expdec} holds for any smooth initial datum.
\end{theorem}
\begin{proof}[Proof of Theorem \ref{lemgenericassp}]
Trivially, from \eqref{user}, \eqref{S}, \eqref{F},  \eqref{cond1},  \eqref{cond2} and recalling the notation \eqref{equivnot}:
$$
(-\cL + \pa_t) \Gamma (f_t) \leq   (\gamma- \mu) \|  \V f_t\|^2
= - \lambda \Gamma (f_t) \,.
$$
\end{proof}
If we divide both sides of  \eqref{cond2} by $\|  \V f_t\|^2\am $, then it becomes clear that imposing  condition \eqref{cond2} is equivalent to requiring that (there exists a bilinear form on $\mathfrak{R}^{\lv \A_m\rv}\simeq \R^{N \lv \A_m\rv} $ such that) the ``angle" between the vectors $\Lambda \V f_t$  and $\V f_t$  is obtuse. 

We now establish conditions under which \eqref{cond1} and \eqref{cond2} hold.

\begin{theorem}\label{lemcond1}
Let  $\{V_i: i=0 \dd d \}$  be the vector fields appearing in \eqref{SDE}. Then the following holds:

\medskip

i) If the vector fields $\{V_i: i=0 \dd d \}$  satisfy  the UFG condition   for some $m \in \N$, then there exists a choice of the constants $\{a_{[\alpha]}\}_{\alpha \in \A_m}$ such that   \eqref{cond1} is satisfied. 

\bigskip

ii) Suppose  the assumption of the above point i) is satisfied  and assume that there exists a real number $\lambda_0>0$ such that, for every  $\alpha \in  \A_m$ and every smooth enough function $f$, 
\be\label{dilcon}
 \left(V_{[\alpha]}f\right) \left( [V_{[\alpha]}, V_0]f\right)  \leq  - \lambda_0
\lv V_{[\alpha]} f\rv^2 \,. 
\ee  
If $\lambda_0$ is big enough then \eqref{cond2} holds with $\mu=\lambda_0$. Hence there exists $\lambda>0$ such that  \eqref{expdec} holds for any smooth initial datum. Therefore, by   Proposition \ref{Bakrystrategy} part (c), if the initial datum $f_0$ is  continuous and bounded, then  for any $t_0\in (0,1)$ and any $K>0$
 there exists a constant $c_{t_0,K}>0$ such that   
$$
\sup_{x \in \mathbb{B}(0,K)}\lv V_{[\alpha]}f_t (x)\rv ^2 \leq c_{t_{0},K} \,  e^{-\lambda (t-t_0)}
\|f_0(x)\|_{\infty} \qquad {\mbox{for all }} \alpha \in \A_m \ \mbox{and all}\quad t>t_{0}. 
$$
 If the coefficients of the vector fields $\{V_0 \dd V_d\}$ are bounded, then 
\be\label{AAAAAAAAA}
\lv V_{[\alpha]}f_t (x)\rv ^2 \leq c_{t_{0}} \,  e^{-\lambda (t-t_0)} \|f_0(x)\|_{\infty} \qquad {\mbox{for all }} x \in \R^N, \alpha \in \A_m \ \mbox{and all}\quad t>t_{0}. 
\ee

\end{theorem}

 
\begin{note} \label{note41}\textup{ Let us clarify the statement of Theorem \ref{lemcond1}. According to part i) of Theorem \ref{lemcond1}, if the UFG condition holds then one can fix a bilinear form $\la\cdot, \cdot \ra\am$ such that \eqref{cond1} holds. In the statement of part ii) of the  theorem  we intend \eqref{cond2} to be satisfied for the same bilinear form. An explicit estimate on how big $\lambda_0$ is will be given  in the proof, see \eqref{conlambda0}. Obviously the estimate \eqref{conlambda0} is quite general and can be made more precise when explicit knowledge of the functions $\varphi$'s appearing in the UFG condition is available. We also remark that \eqref{dilcon} is a slight generalization of the so-called {\em dilation condition}, which has been considered in the literature for elliptic and hypoelliptic semigroups (see \cite[Section 2]{DragKonZeg} and references therein). More  generally, \eqref{dilcon} replaces in a quantitative way the exact dilation structure of stratified Lie groups.  Still regarding \eqref{dilcon}, notice that one cannot expect that the UFG condition alone could yield exponential decay of the derivatives of the semigroup (as we have already pointed out in the introduction, if $\cL$ is uniformly elliptic then it satisfies the UFG condition, but not every elliptic dynamics has derivatives that decay exponentially fast), therefore some quantitative condition on the vector fields has to be imposed. 
}\hf
\end{note}
Let us present some examples of UFG generators that satisfy the assumptions of Theorem \ref{lemcond1}, in particular condition \eqref{dilcon}. 

\begin{example}[UFG-Gru\v sin Plane]\textup{Let $d=1$ and $N=2$, i.e. consider the operator $\cL= V_0+V_1^2$ on $\R^2$, with
$$
V_0= k x \pa_x, \quad V_1= x\pa_y, \qquad k>0.
$$
The fields $\{V_0, V_1\}$ satisfy the UFG condition with $m=1$, as
$[V_1, V_0]= -k V_1$. 
It is easy to see that \eqref{gcon} holds with $\lambda=2 k$ (for every $k>0$). Indeed, by direct calculation:
\begin{align*}
(-\cL + \pa_t) \lv V_1 f_t \rv^2 &\stackrel{\eqref{lem1}}{=}
-2 \lv V_1^2 f_t\rv^2+ 2 ([V_1, \cL] f_t)(V_1 f_t) \\
&=  -2 \lv V_1^2 f_t\rv^2 -2 k  \lv V_1 f_t \rv^2 \leq 
-2 k  \lv V_1 f_t \rv^2  \,.
\end{align*}
We name this example the UFG-Gru\v sin plane as it results from a small modification of the so-called UFG-Gru\v sin plane, given by the operators
$$
X_0:= x\pa_x + 2 y \pa_y, \quad X_1:= \pa_x , \quad X_2:=  x\pa_y. 
$$ 
It is easy to verify that the operator $X_0+ X_1^2 + X_2^2$ verifies the H\"ormander condition. 
}
\end{example}
\begin{example}\textup{
The operators $\{V_0, V_1, V_2\}$ defined in 
Example \ref{exa:Heisenberg} satisfy the assumptions of Theorem \ref{lemcond1}. In particular in this case  one can obtain the following result, the proof of which can be found in Section \ref{sec5}. 
}\hf
\end{example}



\begin{lemma}\label{lem:example}
Let d=2 and consider the operator $\cL$ of the form \eqref{gen1} acting on $\R^3$,  where the fields $V_0, V_1,V_2$ are those defined in Example \ref{exa:Heisenberg}.  With the notation introduced in Example \ref{exa:Heisenberg}, we have that for every $k>0$, \eqref{gcon} holds with $\lambda=k$, i.e.
$$
(-\cL+ \pa_t) \Gamma(f_t) \leq -k \Gamma (f_t). 
$$
\end{lemma}
We include the proof of the above lemma in Section \ref{sec5} for two reasons: i)  to show on a simple example how the proof of Theorem  \ref{lemcond1}  works in practice, without all the cumbersome notation that one needs to prove the result in general; ii) to show that,   thanks to the freedom to choose the constants appearing in the definition of $\Gamma$ (see Note \ref{pc}), 
 the general lower bound  for $\lambda$ given in \eqref{conlambda0} can be improved when we explicitly know the functions $\varphi$'s deriving from the UFG condition. 


\begin{note} \textup{The quadratic form $\Gamma(f_t)$ includes the derivatives of the semigroup but not the semigroup $f_t$ itself. Therefore the results of this paper only give information on  the behaviour of the derivatives; in Subsection \ref{subs41} below we  use such results to obtain some (partial) information on  the asymptotic behaviour of the semigroup $f_t=\cP_tf$. An analogous observation  holds for the derivatives in the direction $V_0$. Notice that our result does not imply anything regarding the behaviour in the direction $V_0$, as $V_0$ is not contained in the definition of $\Gamma$. This is again a structural fact. Indeed, under just the UFG condition, one is not even guaranteed differentiability in the direction $V_0$, let alone decay, see \cite[Section 2.9]{Crisan}. However it was proved that under the so-called {\em $V_0$-condition} (see Definition \ref{vocond} below), the semigroup $\cP_t$ is differentiable in the direction $V_0$ as well. In this case our results cover such a direction as well. 
}\hf
\end{note}

\begin{definition}[$V_0$-condition]\label{vocond} With the notation introduced so far, we say that the $V_0$-condition is satisfied if   there exist functions $\varphi_{\beta} \in C_V^{\infty}$ such that
$$
V_0= \sum_{\beta \in \A_2} \varphi_{\beta}V_{[\beta]}.
$$
\end{definition} 
  \begin{corollary} Suppose that the assumptions of Theorem \ref{lemgenericassp} are satisfied. If the $V_0$ condition holds, then there exist positive constants $c, \lambda > 0$  such that 
$$
\lv V_0 f_t\rv ^2 \leq c e^{-\lambda t},
$$  
say for any smooth $f_0$.
  \end{corollary}

Finally we observe, although without proof, that the same strategy  used in this paper  can be adapted to obtain estimates on the derivatives of any order along the semigroup. This can be done inductively (on the order of the derivative) using, at step $n$ of the induction, the quadratic form 
$$
(\Gamma^{(n)} f)(x):= \sum_{k=1}^n\sum_{\alpha^{(j)} \in \A_m} a_{[\alpha^{(k)}\dd {\alpha^{(1)}]}} \lv V_{[\alpha^{(k)}]} \dots V_{[\alpha^{(1)}]}  f_t(x) \rv^2.
$$
That is, the quadratic form $\Gamma^{(n)}$ contains all the derivatives of  order at most $n$, in all the directions contained in $\mathfrak{R}_m$. 
A similar  inductive procedure has  been used, for hypoelliptic semigroups of hypocoercive type, in \cite{mythesis, MV_I}.

\subsection{Decay of the semigroup}\label{subs41}
Let $x$ and $y$ be two points in $\R^N$ and, for some given $\alpha \in \A_m$, suppose there exists an integral curve of $V_{[\alpha]}$ joining $x$ and $y$. That is, suppose there exists $\eta(\tau):[0,1]\rightarrow \R^N$ such that
$$
\frac{d}{d\tau}\eta(\tau)=V_{[\alpha]} ( \eta(\tau)), \qquad \eta(0)=x, \,\eta(1)=y.
$$
We stress that in the above $V_{[\alpha]}$ has to be intended as a vector field rather than as a differential operator. More generally, we say  that $y$ is reachable from $x$, and write $x \sim y$, if there exists an integer $M>0$ and $M$ points in $\R^N$, $z_1 \dd z_M$, such that $z_1=x$, $z_M=y$ and for every $i=1\dd M-1$, there exists  an  $\alpha^{(i)} \in \A_m$ such that the integral curve of    $V_{[\alpha^{(i)}]}$ is well defined and joins $z_i$ with $z_{i+1}$.   The relation $\sim$ is an equivalence relation. We denote by $\mathcal{U}_x$ the set of points reachable from $x$ (clearly, if $y \sim x$ then $\mathcal{U}_x=\mathcal{U}_y$). 
 
\begin{corollary}\label{corcontrol}
 Let $\cP_t$ be the semigroup \eqref{semigroup} and assume for simplicity that the fields $V_{[\alpha]}$ have bounded coefficients.  Suppose that the assumptions of Theorem \ref{lemcond1} hold.   Then for any $f(x)$ continuous and bounded,   for any $x \in \R^N$ and for any $y \in \mathcal{U}_x$ there exists $\lambda >0$ such that
$$
\lv (\cP_tf)(x)- (\cP_tf)(y)\rv \leq c e^{-\lambda t }, \qquad \mbox{for all }t>0, 
$$
where $c>0$ is a constant independent of $t$. 

\end{corollary}
\begin{proof}
We just need to prove the result for $M=1$ (that is, when $y$ can be reached from $x$ moving along the integral curve of one of the $V_{[\alpha]}$'s).  For any fixed $t>0$, by definition of directional derivative we have 
$$
\frac{d}{d\tau}(\cP_tf)(\eta(\tau)) = (V_{[\alpha]} \cP_tf)( \eta(\tau))
$$
 (see also  Note \ref{notenotenote} in the appendix).  Integrating the above between 0 and 1 and using \eqref{AAAAAAAAA}  we obtain the result.      
\end{proof}

\section{Proofs of Main results}\label{sec5}
Throughout this section, if  $\varphi(x)$ is a function, we denote
$$
\bar{\vp}:= \sup_{x}\lv \vp(x) \rv \,.
$$
We also set
$$
\bar{\A}_m := \{\alpha \in \A: \| \alpha\|= m   \}. 
$$
 We make the (obvious) remark, that  
if the UFG condition holds for some $m \in \N$, then for all the multi-indices $\alpha$ of length at most $m$, we have
\be\label{derphi}
\bar{\vp}^j_{\alpha\ast j, \beta}:= \sup_{x} \lv V_j({\vp}_{\alpha\ast j, \beta}) \rv < \infty, \qquad \forall j=1\dd d.
\ee 
We  also recall the  Young's inequality 
\be\label{yi}
\lv a \, b \rv \leq \frac{a^2}{2 \epsilon}+ \frac{b^2\epsilon}{2}, \quad \mbox{for all } a,b \in \R \mbox{ and } \epsilon>0, 
\ee
which we will repeatedly use throughout the proofs of this section. 
\begin{proof}[Proof of part i) of Theorem \ref{lemcond1}]
The case $m=1$ is straightforward and can be dealt with directly, so throughout the proof we take $m>1$. 
Looking at \eqref{Scom},  notice that if $\|\alpha_k\|=m$ then $\| \alpha_k \ast j\|=m+1$ when $j \in \{1 \dd d\}$, so we can apply the UFG condition to the operator $V_{[\alpha_k \ast j]}$. So from \eqref{Scom} we obtain

\begin{align*}
\mathcal{S}(f_t) &= 
 -2 \sum_{k=1}^m \sum_{\alpha_k} \al{k}\sum_{j=1}^d
\lv V_j \vl{k} \rv^2   \\
& + 4 \sum_{k=1}^{m-1} \sum_{\alpha_k} \al{k} \sum_{j=1}^d
\left( V_j V_{[\alpha_k \ast j]} f_t\right) \left( \vl{k}\right) \qquad \qquad {\star}\\
& + 4 \sum_{\alpha_m} \al{m} \sum_{j=1}^d \sum_{\beta \in \A_m}
\vp_{\alpha_m \ast j, \beta}^j \left( V_{[\beta]} f_t\right) \left( \vl{m}\right)
\qquad \qquad {\ast\ast}\\
& + 4 \sum_{\alpha_m} \al{m} \sum_{j=1}^d \sum_{\beta \in \A_m}
\vp_{\alpha_m \ast j, \beta} \left( V_jV_{[\beta]} f_t\right) \left( \vl{m}\right)
\qquad \qquad {\triangle}
\end{align*}

Let us now set, for any multi-index $\gamma \in \A_m$,  
\begin{align*}
\mathcal{J}_{\gamma}:= \sup_{\substack{j=1 \dd d  \\ \beta \in \A_m}} 
\bar{\vp}_{\gamma \ast j, \beta} \qquad {\mbox{ and }} \qquad
\mathcal{H}_{\gamma}&:= \sup_{\substack{j=1 \dd d  \\ \beta \in \A_m\backslash \gamma}} 
\bar{\vp}_{\gamma \ast j, \beta}^j \,.
\end{align*}
 With these definitions in mind, let us start estimating each of  the above terms, beginning with the last.
\begin{description}
\item [Terms with $\triangle$] For each $\alpha_m \in \bar{\A}_m$,
\begin{align}
&4 \al{m}\sum_{j=1}^d \sum_{\beta \in \A_m} \vp_{\alpha_m \ast j, \beta}
\left(V_j V_{[\beta]} f_t\right) \left( \vl{m}\right)\nonumber\\
& \leq 2 \sum_{j=1}^d \sum_{\beta \in \A_m} \left[
{\bf{1}}_{\{\mathcal{J}_{\alpha_m}\neq 0\}} \lv V_j V_{[\beta]} f_t \rv^2 + \al{m}^2 \mathcal{J}_{\alpha_m}^2 \lv \vl{m}\rv^2
\right]\nonumber\\
&= 2 \al{m}^2 \mathcal{J}_{\alpha_m}^2 d \lv \A_m\rv \lv \vl{m}\rv^2
+ 2 {\bf{1}}_{\{\mathcal{J}_{\alpha_m}\neq 0\}} \sum_{j=1}^d \sum_{\beta \in \A_m} \lv V_j V_{\beta}f_t \rv^2
\nonumber 
\end{align}
\item[Terms $\ast \ast$]For each $\alpha_m \in \bar{\A}_m$, we have
\begin{align}
&4 \al{m}\sum_{j=1}^d \sum_{\beta \in \A_m} \vp_{\alpha_m \ast j, \beta}^j
\left(V_{[\beta]} f_t\right) \left( \vl{m}\right)\nonumber\\
& \leq 4 \al{m}\sum_{j=1}^d \vp_{\alpha_m \ast j, \alpha_m}^j \lv \vl{m}\rv^2 + 4d \mathcal{H}_{\alpha_m} a_{[\alpha_m]}
\sum_{\beta \in \A_m\backslash \{\alpha_m\}}
(V_{[\beta]}f_t )(V_{[\alpha_m]}f_t )\nonumber\\
& \leq 4 \al{m}\sum_{j=1}^d \vp_{\alpha_m \ast j, \alpha_m}^j \lv \vl{m}\rv^2 \nonumber\\
& + 2d {\bf{1}}_{\{\mathcal{H}_{\alpha_m}\neq 0\}} \sum_{\beta \in \A_m\backslash \{\alpha_m\}} \lv V_{[\beta]}f_t \rv^2 
+ 2d \al{m}^2 \mathcal{H}_{\alpha_m}^2 \left( \lv \A_m\rv - 1\right) \lv \vl{m}\rv^2 
\nonumber 
\end{align}

\item[Terms $\star$] for every $k=1 \dd m-1$, 
\begin{align}
&4 \al{k} \sum_{j=1}^d \left( V_j V_{[\alpha_k \ast j]} f_t\right)
\left( \vl{k}\right) \leq 2d \lv \vl{k} \rv^2 + \sum_{j=1}^d 2 \al{k}^2
\lv V_j V_{[\alpha_k \ast j]} f_t \rv^2 
\label{termsstar}
\end{align}
\end{description}
Putting the above estimates together, after setting
$$
\mathcal{J}:= \sum_{\alpha_m \in \bar{\A}_m} {\bf{1}}_{\{\mathcal{J}_{\alpha_{m}}\neq 0\}}, 
$$
 we obtain 
\begin{align}
\mathcal{S} (f_t)&\leq \sum_{k=1}^m
\sum_{\alpha_k \in \bar{\A}_k} c_{\alpha_k}
 \lv \vl{k}\rv^2    
\label{tr1}\\
& + \sum_{k=1}^m \sum_{j=1}^d
\sum_{\alpha_k \in \bar{\A}_k}
\lv V_j\vl{k}\rv^2 \left[ 
-2 \al{k} + 2 \left( \mathcal{J}+  {\bf{1}}_{\{ k>1\}}{\bf{1}}_{\{ \alpha_k= \alpha_{k-1}\ast j\}} 
\al{k-1}^2\right)  \right] \label{t2}
\end{align}
where 
\begin{align*}
c_{\alpha_m}&:= 
2 \al{m}^2 \mathcal{J}_{\alpha_m}^2
d \lv \A_m \rv  
 + 4 \al{m}  \sum_{j=1}^d \varphi_{\alpha_m  \ast j, \alpha_m}^j
+ 2d \!\!\! \!\!\!\sum_{\beta_m \in \bar{\A}_m \setminus \{\alpha_m\}}\!\!\!\!\!\!\!\!{\bf{1}}_{\{\mathcal{H}_{\beta_{m}}\neq 0\}}
+ 2d \al{m}^2 \mathcal{H}^2_{\alpha_m}\left(\lv \A_{m} \rv -1 \right)
\end{align*}
and, for $k=1 \dd m-1 $
\begin{align*}
c_{\alpha_k}:= 2d 
+ 2d \sum_{\alpha_m \in \bar{\A}_m} \!\!\! {\bf{1}}_{\{\mathcal{H}_{\alpha_{m}}\neq 0\}}\,.
\end{align*}
With the purpose of making sure that the terms in \eqref{t2} are negative we can simply choose
\be\label{cc1}
\al{1}> \max\{0,\mathcal{J}\}, \quad \mbox{and} \quad \al{k}> \mathcal{J}+ \al{k-1}^2 \qquad \mbox{for all }k=2 \dd m\,.
\ee
Therefore, once all the $\al{1}$ have been fixed, all the other coefficients can be chosen through the above recursive relation. This choice allows to fix all the constants in the expression for the quadratic form  $\Gamma$. Assuming that any choice satisfying \eqref{cc1} has been made, one then has

\begin{align*}
\mathcal{S}(f_t)&\leq \sum_{k=1}^m\sum_{\alpha_k \in \bar{\A}_k} c_{\alpha_k}    \lv V_{[\alpha_k]}f_t \rv^2     \leq \gamma \| \V f_t\|^2, 
\end{align*}    
having set 
$$
\gamma:= \max_{\substack{ \alpha_k \in \A_k \\k=1 \dd m }}     \frac{c_{\alpha_k}}{a_{[\alpha_k]}} \,.
$$
\end{proof}


\begin{proof}[Proof of part ii) of Theorem \ref{lemcond1}]
For simplicity, suppose $m>3$. The case $m \leq 3$ can be studied analogously (and it is in fact less involved). 
We notice again that if $\| \alpha_k\|=m-1$ ($m$, respectively) then 
$\|\alpha_k \ast j \ast j\|=m+1$ ($m+2$, respectively). Therefore we can again apply the UFG condition to the vector fields  $V_{[\alpha_k \ast j \ast j]}$ (appearing in \eqref{Scom}-\eqref{S}) when $\|\alpha_k\|=m-1$ or $m$, obtaining
\begin{align*}
2 \la \Lambda \V  f_t, \V f_t  \ra & \stackrel{\eqref{dilcon}}{\leq}  -2 \sum_{k=1}^m \sum_{\alpha_k} \al{k} \lambda_0 \lv \vl{k}\rv^2 \\
& +2 \sum_{k=1}^{m-2} \sum_{\alpha_k} \al{k} \sum_{j=1}^d
\left( V_{[\alpha_k \ast j \ast j]} f_t\right) \left( \vl{k}\right) 
\qquad \qquad {\diamond}\\
& +2 \sum_{\alpha_{m-1}\in \bar{\A}_{m-1}} \al{m-1} \sum_{j=1}^d
\sum_{\beta \in \A_m} \vp_{\alpha_{m-1}\ast j \ast j, \beta}
\left( V_{[\beta]} f_t\right) \left( \vl{m-1}\right) \qquad \qquad {\Box}\\
& + 2 \sum_{\alpha_m \in \bar{\A}_{m}} \al{m} \sum_{j=1}^d \sum_{\beta \in \A_m}
\vp_{\alpha_m \ast j \ast j, \beta} \left( V_{[\beta]} f_t\right) \left( \vl{m}\right)
\qquad \qquad {\ast}
\end{align*}
Like in the proof of part i) of Theorem \ref{lemcond1}, we set 
$$
\mathcal{I}_{\gamma}:= \sup_{\substack{j=1 \dd d  \\ \beta \in \A_m\backslash \gamma}} \bar{\vp}_{\gamma \ast j \ast j, \beta} 
$$
and estimate all the above terms, starting from the last. 
\begin{description}
\item [Terms with $\ast$] For each $\alpha_m \in \bar{\A}_m$, 
\begin{align}
& 2 \al{m}\sum_{j=1}^d \sum_{\beta \in \A_m} \vp_{\alpha_m \ast j \ast j, \beta}
\left( V_{[\beta]} f_t\right) \left( \vl{m}\right)\nonumber\\
& \leq 2 \al{m}\sum_{j=1}^d  \vp_{\alpha_m \ast j \ast j, \alpha_m}
 \lv \vl{m}\rv^2+ 2 d \mathcal{I}_{\alpha_m} \al{m}   \sum_{\beta \in \A_m\backslash \{\alpha_m\}} \left(V_{[\beta]} f_t \right) \left( \vl{m}\right)\nonumber\\
& = 2 \al{m}\sum_{j=1}^d  \vp_{\alpha_m \ast j \ast j, \alpha_m}
 \lv \vl{m}\rv^2 \nonumber\\
&+ d  \sum_{\beta \in \A_m\backslash \{\alpha_m\}}\left[
\al{m}^2 \mathcal{I}_{\alpha_m}^2 \lv \vl{m} \rv^2  + \lv V_{[\beta]}f_t\rv^2 
{\bf{1}}_{\{\mathcal{I}_{\alpha_m}\neq 0\}} \right]\nonumber\\
&= 2 \al{m}\sum_{j=1}^d  \vp_{\alpha_m \ast j \ast j, \alpha_m}
 \lv \vl{m}\rv^2 \nonumber\\
&+ d \al{m}^2 \mathcal{I}_{\alpha_m}^2 \lv \vl{m} \rv^2  \left( \lv \A_m\rv -1\right)
+ d\, {\bf{1}}_{\{\mathcal{I}_{\alpha_m}\neq 0\}} \sum_{\beta \in \A_m\backslash \{\alpha_m\}}\lv V_{[\beta]}f_t\rv^2
\nonumber 
\end{align}
 
\item[Terms $\Box$]For each $\alpha_{m-1} \in \bar{\A}_{m-1}$, 
\begin{align}
& 2 \al{m-1}\sum_{j=1}^d \sum_{\beta \in \A_m}\vp_{\alpha_{m-1}\ast j \ast j, \beta}
\left(V_{[\beta]}f_t \right) \left(\vl{m-1} \right) \nonumber\\
& \leq  2 \al{m-1}\sum_{j=1}^d \vp_{\alpha_{m-1}\ast j \ast j, \alpha_{m-1}}
\lv \vl{m-1} \rv^2 \nonumber\\
&+ d \al{m-1}^2 \mathcal{I}^2_{\alpha_{m-1}}
\left( \lv \A_m\rv-1\right) \lv \vl{m-1}\rv^2
+ d {\bf{1}}_{\{\mathcal{I}_{\alpha_{m-1}}\neq 0\}}
\sum_{\beta \in \A_m\backslash\{\alpha_{m-1}\}} \lv V_{[\beta]} f_t \rv^2 \,.
\nonumber 
\end{align}

\item[Terms $\diamond$] for every $k=1 \dd m-2$, 
\begin{align}
&2 \al{k} \sum_{j=1}^d \left( V_{[\alpha \ast j \ast j]} f_t\right)
\left( \vl{k} \right)
\leq d \lv \vl{k}\rv^2 + \sum_{j=1}^d \al{k}^2 \lv V_{[\alpha_k \ast j \ast j]} f_t \rv^2  \,.
\label{termsdiamond}
\end{align}
\end{description}

Overall one obtains: 
\begin{align}
2 \la \Lambda \V  f_t, \V f_t  \ra & \leq 
\sum_{k=1}^m
\sum_{\alpha_k \in \bar{\A}_k}
 \lv \vl{k}\rv^2   \left( -2 \al{k} \lambda_0 + \ell_{\alpha_k}\right)
\label{t1}
\end{align}
where
\begin{align*}
\ell_{\alpha_m}&:= d \al{m}^2 \mathcal{I}_{\alpha_m}^2 \left(\lv \A_{m} \rv -1 \right)
+ d \!\!\!\sum_{\beta_m \in \bar{\A}_m \setminus\{\alpha_m\}}\!\!\! {\bf{1}}_{\{\mathcal{I}_{\alpha_{m}}\neq 0\}} \\
&+ 2 \al{m}   \sum_{j=1}^d \varphi_{\alpha_m \ast j \ast j, \alpha_m}+ d \!\!\!\sum_{\alpha_{m-1} \in \bar{\A}_{m-1} }\!\!\!{\bf{1}}_{\{\mathcal{I}_{\alpha_{m-1}}\neq 0\}}+ {\bf{1}}_{\{\alpha_m=\alpha_{m-2} \ast j \ast j\} } \al{m-2}^2
\\
\ell_{\alpha_{m-1}}&:=  2 \al{m-1} \sum_{j=1}^d 
\vp_{\alpha_{m-1} \ast j \ast j, \alpha_{m-1}}
+ d\al{m-1}^2 \mathcal{I}^2_{\alpha_{m-1}} \left( 
\lv \A_{m}\rv -1\right) \\
&+ {\bf{1}}_{\{ \alpha_{m-1}= \alpha_{m-3} \ast j \ast j \}} \al{m-3}^2  +d \!\!\!\sum_{\alpha_{m} \in \bar{\A}_{m} }\!\!\! {\bf{1}}_{\{\mathcal{I}_{\alpha_{m}}\neq 0\}}+ d  \sum_{\alpha_{m-1} \in \bar{\A}_{m-1} } {\bf{1}}_{\{\mathcal{I}_{\alpha_{m-1}}\neq 0\}}
\end{align*}
and, for $k=1 \dd m-2 $
\begin{align*}
\ell_{\alpha_k}:= d + {\bf{1}}_{\{k-2>0\}}
{\bf{1}}_{\{\alpha_k= \alpha_{k-2} \ast j \ast j \}}\al{k-2}^2
+d \!\!\!\sum_{\alpha_{m} \in \bar{\A}_{m} }\!\!\!  {\bf{1}}_{\{\mathcal{I}_{\alpha_{m}}\neq 0\}}
+ d \!\!\!\sum_{\alpha_{m-1} \in \bar{\A}_{m-1} }\!\!\!{\bf{1}}_{\{\mathcal{I}_{\alpha_{m-1}}\neq 0\}}\,.
\end{align*}
 Looking at \eqref{t1}, we then impose
\be\label{cc}
 -2 \al{k} \lambda_0 + \ell_{\alpha_k} \leq - \al{k} \lambda_0 \qquad \forall k=1\dd m\, ,
\ee
that is, 
$$
\al{k} \lambda_0 \geq \ell_{\alpha_k}, \qquad \forall k=1\dd m\,.
$$
It is clear that given any two sets of positive constants, $\al{k}$ and $\ell_{\alpha_k}$, there always exists at least one $\lambda_0>0$ satisfying the above. In particular one can choose any $\lambda_0$ such that 
\be\label{conlambda0}
\lambda_0 > \max_{\substack{{\alpha_k \in \bar{\A}_k}\\{k=1 \dd m}}}
\frac{\ell_{\alpha_k}}{\al{k}}.
\ee
If $\lambda_0$ satisfies \eqref{conlambda0}, and hence \eqref{cc}, from \eqref{t1} ne has

\begin{align*}
2 \la \Lambda \V  f_t, \V f_t  \ra & \leq  -\lambda_0
\sum_{k=1}^m
\sum_{\alpha_k \in \bar{\A}_k} \al{k}
 \lv \vl{k}\rv^2 = -\lambda_0 \| \V f_t\|^2.
\end{align*}
This concludes the proof. 
\end{proof}


\begin{proof}[Proof of Lemma \ref{lem:example}] Consider the quadratic form 
$$
\Gamma(f_t)= a_1 \lv V_1 f_t\rv^2+
a_2 \lv V_2 f_t\rv^2+ a_{12} \lv V_{12} f_t\rv^2\,.
$$
From Proposition \ref{Bakrystrategy} and Proposition \ref{Bakrystrategy2} it is clear that we only need to show  the inequality $(-\cL+ \pa_t) \Gamma(f_t) \leq -k \Gamma(f_t)$. Using \eqref{lem1}, let us therefore calculate the following:
\begin{align*}
(-\cL+ \pa_t) \Gamma(f_t) &= -2a_1 \lv V_1^2f_t\rv^2+ 2a_1 
\left([V_1, \cL] f_t \right) \left( V_1 f_t\right) - 2 a_1 \lv V_2 V_1 f_t \rv^2\\
& -2a_2 \lv V_2^2f_t\rv^2+ 2a_2 
\left([V_2, \cL] f_t \right) \left( V_2 f_t\right) - 2 a_2 \lv V_1 V_2 f_t \rv^2\\
  &-2a_{12} \lv V_1 V_{12} f_t\rv+ 2a_{12} 
\left([V_{12}, \cL] f_t \right) \left( V_{12} f_t\right) 
- 2 a_{12} \lv V_2 V_{12} f_t \rv^2\,.
\end{align*}
The commutators appearing in the above can be calculated, and they are
\begin{align*}
[V_1, \cL] &= -k V_1 + 2 V_2 V_{12}\\
[V_2, \cL] &= -k V_2 + 2 V_1 V_{21}=-k V_2 - 2 V_1 V_{12}\\
[V_{12}, \cL] &= -2 k V_{12} ,.
\end{align*} 
Therefore
\begin{align*}
(-\cL+ \pa_t) \Gamma(f_t) &= -2a_1 \lv V_1^2f_t\rv^2- 2 a_1 \lv V_2 V_1 f_t \rv^2 - 2a_1 k \lv V_1f_t\rv^2 \\
& -2a_2 \lv V_2^2f_t\rv^2- 2 a_2 \lv V_1 V_2 f_t \rv^2 - 2a_2 k \lv V_2f_t\rv^2\\
& - 2 a_{12} \lv V_1 V_{12} f_t\rv^2- 2 a_{12} \lv V_2 V_{12} f_t\rv^2
-4 a_{12}k \lv V_{12}f_t \rv^2\\
& + 4 a_{1} \left(V_2 V_{12} f_t\right) \left( V_1f_t \right)
- 4 a_2 \left(V_1 V_{12} f_t \right)  \left( V_2 f_t\right)\,.
\end{align*}
If we use the Young's inequality \eqref{yi} (with $\epsilon=a_1$ for the first inequality and $\epsilon=a_2$ in the second) we  can estimate the terms on the last line as
\begin{align}
4 a_{1} \left(V_2 V_{12} f_t\right) \left( V_1f_t \right)
& \leq  2 a_1^2\lv V_2 V_{12} f_t\rv^2+ {2} \lv V_1 f_t\rv^2 \label{yae1}\\
4 a_{2} \left(V_1 V_{12} f_t\right) \left( V_2f_t \right)
& \leq 2 a_2^2\lv V_1 V_{12} f_t \rv^2 + {2}\lv V_2 f_t \rv^2 \label{yae2}
\end{align}
Therefore,  
\begin{align*}
(-\cL+ \pa_t) \Gamma(f_t) & \leq \left(- 2a_1 k+ {2}\right)  \lv V_1 f_t\rv^2
+  \left(- 2a_2 k+ {2}\right)  \lv V_2 f_t\rv^2
- 4 a_{12}k \lv V_{12}f_t \rv^2 \quad \spadesuit\\
& -2a_1 \lv V_1^2f_t\rv^2 - 2 a_1 \lv V_2 V_1 f_t \rv^2 -2a_2 \lv V_2^2f_t\rv^2- 2 a_2 \lv V_1 V_2 f_t \rv^2\\
& \left( -2a_{12}+ 2 a_2^2 \right) \lv V_1 V_{12} f_t \rv^2
+ \left( -2a_{12}+ 2 a_1^2 \right) \lv V_2 V_{12} f_t \rv^2 \qquad\qquad \heartsuit
\end{align*}
Looking at the terms $\spadesuit$, we choose $a_1$ and $a_2$ such that  
\be\label{c1}
- 2a_i k+ {2} \leq  -a_i k \,\,\,  \Longrightarrow  \,\,\,     a_i \geq \frac{2}{k} \quad i=1,2
\ee
and $a_{12}$ such that $-4a_{12}k \leq -a_{12}k$, which is true e.g. for any $a_{12}>1$. 
Then, looking at the terms $\heartsuit$, we choose $a_{12}$ much  bigger than $a_1$ and $a_2$, more precisely we choose $a_{12}$ such that 
$$
-2 a_{12}+ 2 a_i^2 < 0, \,\,\, i=1,2.
$$
Because for any $k>0$ one can find $a_1>0$ and $a_2>0$ such that \eqref{c1} is satisfied, this concludes the proof.  
\end{proof}
\begin{note}\label{pc}\textup{If the constants $a_1, a_2, a_{12}$ had not been introduced, i.e. if  $a_1=a_2=a_{12}=1$, then we would have only been able to prove the result for $k > 1/2$ (by making better use of the Young inequality in \eqref{yae1} and \eqref{yae2}). }
\end{note}

\section{Appendix}
We define here the notion of classical solution  $u$ of the PDE (\ref{linearPDE}). The notion  is quite natural: we will require $u$ to be continuously differentiable (twice) in the direction of every vector field $V_i$, $i=1,...,d$. As a consequence of the need that $u$ satisfies (\ref{linearPDE}), we will also  require $u$  to be continuously differentiable in the direction ${\mathcal V}_0=\partial_t-V_0$, when viewed as a function 
$(t,x)\mapsto u(t,x)$ over the  product space 
$(0,\infty)\times {\mathbb R}^N$\(. \text{ }\)However we will not require $u$ to be differentiable in either the time direction $\partial_t$ or the direction $V_0$. 

The analysis of $u$ hinges on being well approximated by solutions of the PDE (\ref{linearPDE}) with smooth initial condition (and therefore smooth for all $t\ge 0$). The approximation is done in such way that, in the limit, only the  differentiability in the directions $V_i$ $i=1,...,d$          \ and ${\mathcal V}_0=\partial_t-V_0$ is preserved, but not that in the time direction $\partial_t$ or in the direction $V_0$. This is to be expected as  the smoothing   effect only takes place in the directions  $V_i$, $i=1,...,d$. An extreme case where the UFG condition holds is  when all the $V_i$, $i=1,...,d$ are equal to zero. Take, for example, the transport equation 
\begin{align*}
 \pa_t u(t,x) &=V_0 u(t,x)\\
 u(0,x)& =f(x).
\end{align*}
For example, assume that $N=1$, $V_0={\partial \over \partial x}$. In this case, the solution is explicitly given by $u(t,x)=f(t+x)$,  
$x\in \R$ and $t\ge 0$. Obviously should $f$ not be differentiable (choose for example $f(x)=|x|$,  $x\in\R$), we will not expect differentiability in either the time direction $\partial_t$, or the space direction 
$\partial \over \partial x$. However $u$ will be differentiable in the direction
${\mathcal V}_0=\partial_t-V_0$. In fact, $u$ is constant in the direction ${\mathcal V}_0=\partial_t-V_0$, as ${\mathcal V}_0u=0$. In this extreme case, no additional smoothness is gained because of the absence of any second order differential operator in the PDE (\ref{linearPDE}). 

At the other end of the spectrum we have the case when the vector fields $V_i$,  $i=0,...,d$,  satisfy the H\"{o}rmander condition. In this case the smoothing effect  occurs in every direction. In particular, $u$ becomes differentiable in the $V_0$ direction, and since $u$ is differentiable in the direction ${\mathcal V}_0=\partial_t-V_0$,
$u$ will also be differentiable in the time direction. In this case,  the notion of a classical solution defined below coincides with the standard notion of a classical solution. 

Finally, we remark that the H\"{o}rmander condition is not necessary to ensure that  $u$ becomes differentiable in the $V_0$ direction (and therefore  also in the time direction). If the vector fields $V_i$ $i=0,...,d$ satisfy the UFG condition and  $V_0$ belongs to $\mathcal{A}$
\footnote{For example, if $V_0$ is a linear combination of the vector fields $V_i, [V_i,V_j],$ $i,j=1,...,d$. }, then it is still the case that $u$ becomes differentiable in the $V_0$ direction and  in the time direction.

 To introduce rigourously the classical solution of the PDE (\ref{linearPDE}) we need several spaces of functions, which we come to introduce. For an open ball ${\mathbb B} \subset \R^N$ and for a function 
$\varphi$ in ${C}^{\infty}_V(\mathbb{B}) $ (that is, for any smooth bounded real-valued function $\varphi$ with bounded derivatives on ${\mathbb B}$ of any order in the directions $V_{[\alpha]}$, $\alpha \in \A_m$), we set
\[
\|\varphi\|_{{\mathbb B},\infty}^{V,1} = \|\varphi\|_{{\mathbb B},\infty} + \sum_{\alpha \in {\mathcal A}_m}
\|V_{[\alpha]} \varphi\|_{{\mathbb B},\infty}
\] 
and then define 
${\mathcal D}_{V}^{1,\infty}({\mathbb B})$ as the closure of  $\mathcal{C}_V^{\infty}(\mathbb{B})$ in ${\mathcal C}_b(\bar{{\mathbb B}})$ w.r.t. 
$\| \cdot \|_{{\mathbb B},\infty}^{V,1}$.\footnote{Notice that this closure is well defined, see \cite[Section 2.3]{CrisanDelarue}} More generally, for $k> 1$, we can define by induction
\begin{equation*}
\|\varphi\|_{{\mathbb B},\infty}^{V,k} = \|\varphi\|_{{\mathbb B},\infty}^{V,k-1} + \sum_{\alpha_1,\dots,\alpha_k \in {\mathcal A}_m}
\|V_{[\alpha_1]} \dots V_{[\alpha_k]} \varphi\|_{{\mathbb B},\infty}, \quad 
\varphi \in {\mathcal C}^{\infty}_V({\mathbb B}).
\end{equation*}
We then define ${\mathcal D}_{V}^{k,\infty}({\mathbb B})$ as the closure
of ${\mathcal C}^{\infty}_V({\mathbb B})$ in ${\mathcal C}_b(\bar{{\mathbb B}})$ w.r.t. 
$\| \cdot \|_{{\mathbb B},\infty}^{V,k}$.  In particular, we can define
${\mathcal D}_{V}^{k,\infty}({\mathbb R}^N)$ as 
\begin{equation*}
{\mathcal D}_{V}^{k,\infty}({\mathbb R}^d) = \bigcap_{r \geq 1} {\mathcal D}_{V}^{k,\infty}({\mathbb B}(0,r)),  
\end{equation*}
where ${\mathbb B}(0,r)$ stands for the $d$-dimensional ball of center $0$ and radius $r$.   
For $v \in {\mathcal D}_{V}^{k,\infty}(\R^N)$,
$V_{[\alpha_1]} \dots V_{[\alpha_k]} v$ is understood as the derivative of $v$ in the directions $V_{[\alpha_1]} \dots V_{[\alpha_k]}$, with $\alpha_1,\dots,\alpha_k \in {\mathcal A}_m$.  Similarly, for $\varphi \in {\mathcal C}^{\infty}_V({\mathbb B})$ and
$k\geq 0$, we set
\begin{equation*}
\|\varphi\|_{{\mathbb B},\infty}^{V,k+1/2} = \|\varphi\|_{{\mathbb B},\infty}^{V,k} + \sum_{i=1}^N \sum_{\alpha_1,\dots,\alpha_k \in {\mathcal A}_m}
\|V_{[\alpha_1]} \dots V_{[\alpha_k]} V_i \varphi\|_{{\mathbb B},\infty}.
\end{equation*}
(Above, $\|\cdot\|_{{\mathbb B},\infty}^{V,0} = \|\cdot\|_{{\mathbb B},\infty}$.)
We then define ${\mathcal D}_{V}^{k+1/2,\infty}({\mathbb B})$ as the closure
of ${\mathcal C}^{\infty}_V({\mathbb B})$ in ${\mathcal C}_b(\bar{{\mathbb B}})$ w.r.t. 
$\| \cdot \|_{{\mathbb B},\infty}^{V,k+1/2}$ and we set
\begin{equation*}
{\mathcal D}_{V}^{k+1/2,\infty}({\mathbb R}^N) = \bigcap_{r \geq 1} {\mathcal D}_{V}^{k+1/2,\infty}({\mathbb B}(0,r)),\ \  k\ge 0. 
\end{equation*}
\begin{note}\label{notenotenote}\textup{
Note that any function in ${\mathcal D}_{V}^{1,\infty}({\mathbb R}^N)$ is differentiable along the solutions of the ordinary differential equation $\dot{\gamma}_t = V(\gamma_t)$, $t \geq 0$, for $V \in {\mathcal A}_m$. In particular, any function in ${\mathcal D}_{V}^{1,\infty}({\mathbb R}^N)$ is continuously differentiable on ${\mathbb R}^N$ when the uniform H\"ormander condition is satisfied.}
\end{note}
To define the notion of a classical solution to \eqref{linearPDE}, we will need to introduce the set of functions that are continuously differentiable in the direction ${\mathcal V}_0=\partial_t  - V_0 $. Again, we proceed via a closure argument.  
For any $r \geq 1$ and any time-space function 
$\varphi \in {\mathcal C}^{\infty}_V([1/r,r] \times {\mathbb B}(0,r))$ with bounded
derivatives of any order, we set 
$$\|\varphi\|_{[1/r,r] \times {\mathbb B}(0,r),\infty}^{{\mathcal V}_0,1}
= \|\varphi\|_{[1/r,r] \times {\mathbb B}(0,r),\infty} + \|{\mathcal V}_0\varphi \|_{[1/r,r] \times {\mathbb B}(0,r),\infty}.$$ 
We then define ${\mathcal D}^{1,\infty}_{{\mathcal V}_0}([1/r,r]
\times {\mathbb B}(0,r))$ as the closure of ${\mathcal C}([1/r,r] \times {\mathbb B}(0,r))$ w.r.t. $\|\cdot \|_{[1/r,r] \times {\mathbb B}(0,r),\infty}^{{\mathcal V}_0,1}$ and then define
${\mathcal D}^{1,\infty}_{{\mathcal V}_0}((0,+\infty)
\times \R^N)$ as the intersection of the spaces 
${\mathcal D}^{1,\infty}_{{\mathcal V}_0}([1/r,r]
\times {\mathbb B}(0,r))$ over $r \geq 1$. We are now in position to define a classical solution to the PDE (\ref{linearPDE})

\begin{definition}
\label{def:semilinearPDE}
We call a function $v=\{v(t,x),(t,x)\in [0,+\infty) \times{\mathbb R}^N\}$ a \emph{classical} solution of the PDE (\ref{linearPDE}) if the following conditions are satisfied:
\begin{enumerate}
\item $v$ belongs to ${\mathcal D}_{{\mathcal V}_0}^{1,\infty}((0,+\infty) \times \R^N)$
and, for any $t>0$, $v(t,\cdot)$ is in ${\mathcal D}_V^{2,\infty}(\R^N)$ ; moreover,  for any $\alpha_1,\alpha_2 \in {\mathcal A}$,  the function 
$(t,x) \in (0,+\infty) \times \R^N
\mapsto \bigl(V_{[\alpha_1]}v(t,x),V_{[\alpha_1]}V_{[\alpha_2]}v(t,x)\bigr)$ is continuous. 
\item For any $(t,x) \in (0,+\infty) \times \R^N$, it holds
$${\mathcal V}_0v(t,x)=  \sum_{i=1}^N V_i^{2} v(t,x).$$
\item The boundary condition $\lim_{(t,y)\to(0,x)}v(t,y)=f(x)$ holds as well for any $x\in {\mathbb R}^N$.
\end{enumerate}
\end{definition}
\begin{note}\label{ODEnote}\textup{
Again, we emphasize that we do not assume that a classical solution of the 
PDE (\ref{linearPDE}) must be differentiable in the time direction or in the direction $V_0$. However this is the case if vector fields satisfy the uniform H\"ormander condition. In this case the above definition coincides with the standard definition of a classical solution.}
\end{note}
The following proposition is a particular case of Proposition 2.8 in \cite{CrisanDelarue}:
\begin{prop}
\label{prop:11:4:2}
Under the UFG condition, if $f$ is a continuous function of polynomial growth, the function 
$(t,x)\mapsto (\cP_t f)(x)$
is a classical solution to the PDE \eqref{linearPDE} in the sense of Definition \ref{def:semilinearPDE}.  
Moreover, any other classical solution $v$ of the linear PDE (\ref{linearPDE}) that has polynomial growth matches the solution $(t,x)\mapsto (\cP_t f)(x)$.
\end{prop}

\begin{lemma}\label{densitylemma}
With the notation introduced so far, if \eqref{glcon} holds for any  $g_0 \in \mathcal{C}^{\infty}_V$ then \eqref{local1} holds for any $g_0 \in\mathcal{D}_{V}^{1,\infty} \cap Pol$.
\end{lemma}
\begin{proof} This is a standard density argument, so we just sketch it. 
Let $g^n_0$ be a sequence in $\mathcal{C}^{\infty}_V$, such that $g^n_0 \stackrel{\|\cdot \|_{\infty}^{V,1}}{\longrightarrow} g_0 \in \mathcal{D}_{V}^{1,\infty}$. Then 
$g^n_0$ and $V_{[\alpha]}g^n_0$ converge uniformly on compacts to $g_0$  and $V_{[\alpha]}g_0$, respectively, for any $\alpha \in \A_m$.   Applying \eqref{glcon} to the sequence $g_0^n$ gives
\be\label{apprfn}
\lv V_{[\alpha]}g_t^n (x)\rv ^2 \leq e^{-\lambda t} { (\cP_t\Gamma g_0^n) (x)}
  \qquad {\mbox{for all }} \alpha \in \A_m, t\ge 0 \,,
\ee
where in the above $g_t^n:= \cP_tg_0^n$.
From the integration by parts formulae in \cite[Chapter 3]{Nee} the left hand side of the above converges uniformly on compacts to $V_{[\alpha]}g_t (x)$. As for the right hand side, since $V_{[\alpha]}g^n_0$ converges uniformly on compacts to $V_{[\alpha]}g_0$, then $(\Gamma g^n_0)(x)$ converges uniformly on compacts, and therefore pointwise, to $(\Gamma g_0)(x)$. By definition of $\cP_t$, we have 
 \begin{align*}
\lv \left[ \cP_t(\Gamma g_0^n- \Gamma g_0)\right](x)\rv=\int \left[(\Gamma g_0^n) (y)- (\Gamma g_0)(y)\right] p_t(x, dy),
\end{align*}
where $p_t(x, dy)$ are the transition probabilities of the process  $X_t$ in \eqref{SDE}. Because $g_0 \in Pol$, we can always choose the approximating sequence so that $\Gamma g_0^n$  grows polynomially  (with the degree of the polynomial independent of $n$). Therefore, by the dominated convergence theorem, 
$(\cP_t\Gamma g_0^n)(x)$ converges pointwise to $(\cP_t\Gamma g_0)(x)$ as $n\rightarrow \infty$. Taking the (pointwise) limit as $n\rightarrow\infty$ on both sides of \eqref{apprfn} gives then 
\begin{align*}
\lv V_{[\alpha]}g_t (x)\rv ^2 & \leq e^{-\lambda t} { (\cP_t\Gamma g_0) (x)}\\
& = e^{-\lambda t} \int (\Gamma g_0)(y) p_t(x,dy)\\
&\stackrel{\eqref{polcon}}\leq  \lv \A_m\rv  \kappa  e^{-\lambda t} \int (1+ \lv y\rv^{q}) p_t(x,dy).
\end{align*}
Now $\int (1+ \lv y\rv^q) p_t(x,dy)= (\cP_t h)(x)$, where $h(x)=1+\lv x\rv^q$ and therefore, by Proposition \ref{prop:11:4:2}, $(\cP_t h)(x)$ is polynomially bounded. Taking the supremum over compact sets on both sides of  the above gives the desired result. 

\end{proof}

{\bf Data accessibility} There are no data in this paper
\smallskip

{\bf Competing Interests} We have no competing interests
\smallskip

{\bf  Authors' Contributions} Both authors carried out the work described in this paper and gave final approval for this manuscript.
\smallskip

{\bf  Funding Statement} No grants or funding to acknowledge

\end{document}